\documentclass[review,12pt,a4paper,3p]{elsarticle}

\usepackage{amsmath}
\usepackage{lipsum,bm}
\usepackage{amsfonts,wasysym}
\usepackage{graphicx}
\usepackage{mathrsfs}
\usepackage{epstopdf}
\usepackage{algorithmic}
\usepackage{amsopn}
\usepackage{amssymb}
\usepackage{enumitem}
\usepackage{url}
\usepackage{hyperref}
\usepackage{subfig}
\usepackage{xcolor}

\ifpdf
\DeclareGraphicsExtensions{.eps,.pdf,.png,.jpg}
\else
\DeclareGraphicsExtensions{.eps}
\fi

\newdefinition{example}{Example}[section]
\newdefinition{setting}{Setting}[section]
\newdefinition{framework}{Framework}[section]
\newdefinition{rem}{Remark}[section]
\newdefinition{defn}{Definition}[section]

\newtheorem{assum}{Assumption}[section]
\newtheorem{thm}{Theorem}[section]

\newtheorem{lem}{Lemma}[section]
\newproof{proof}{Proof}

\numberwithin{equation}{section}

\newcommand{\norm}[1]{\left\Vert#1\right\Vert}
\newcommand{\abs}[1]{\left\vert#1\right\vert}
\newcommand{\set}[1]{\left\{#1\right\}}
\newcommand{\seq}[1]{\left<#1\right>}
\newcommand{\bra}[1]{\left[#1\right]}

\newcommand{\E}[1]{\mathbb{E}\left[#1\right]}

\numberwithin{equation}{section}
\definecolor{softred}{RGB}{178, 34, 34}

\begin{document}
	\begin{frontmatter} 
		\title{Deep Learning Based on Randomized Quasi-Monte Carlo Method for Solving Linear Kolmogorov Partial Differential Equation}

		\author[1]{Jichang Xiao\corref{cor1}}
		\ead{xiaojc19@mails.tsinghua.edu.cn}
		\address[1]{Department of Mathematical Sciences, Tsinghua University,  Beijing 100084, China}
		
		\author[2]{Fengjiang Fu}
		\ead{ffj18@mails.tsinghua.edu.cn}
		
		\address[2]{Department of Mathematical Sciences, Tsinghua University,  Beijing 100084, China}
		
		\author[3]{Xiaoqun Wang}
		\ead{wangxiaoqun@mail.tsinghua.edu.cn}
		
		\address[3]{Department of Mathematical Sciences, Tsinghua University,  Beijing 100084, China}

		\cortext[cor1]{Corresponding author}
		\begin{abstract}
			Deep learning algorithms have been widely used to solve linear Kolmogorov partial differential equations~(PDEs) in high dimensions, where the loss function is defined as a mathematical expectation. We propose to use the randomized quasi-Monte Carlo (RQMC) method instead of the Monte Carlo (MC) method for computing the loss function. In theory, we decompose the  error from empirical risk minimization~(ERM) into the generalization error and the approximation error. Notably, the approximation error is independent of the sampling methods. We prove that the convergence order of the mean generalization error for the RQMC method is $O(n^{-1+\epsilon})$ for arbitrarily small $\epsilon>0$, while for the MC method it is $O(n^{-1/2+\epsilon})$ for arbitrarily small $\epsilon>0$. Consequently, we find that the overall error for the RQMC method is asymptotically smaller than that for the MC method as $n$ increases. Our numerical experiments show that the algorithm based on the RQMC method consistently achieves smaller relative $L^{2}$ error than that based on the MC method.
		\end{abstract}
		
	\end{frontmatter}
	
	\section{Introduction}
	Partial differential equations~(PDEs) are important mathematical models to solve problems arising in science, engineering and finance. Classical numerical methods for PDEs are usually based on space partitioning, such as finite differences~\cite{Thomas:2013} and finite elements~\cite{Braess:2007finite}. For high-dimensional PDEs, these methods suffer from the curse of dimensionality, which means that the computational cost grows exponentially as the dimension increases.
	
	Recently, deep learning has achieved considerable progress. Its application to solve PDEs has attracted much attention. Based on minimizing the residual in $L_{2}$-norm, Physics-informed Neural Networks~\cite{Raissi:2019physics} and Deep Galerkin Method~\cite{Sirignano:2018} are proposed. Combining the Ritz method with deep learning, Deep Ritz Method~\cite{Yu:2018} is proposed to solve elliptic PDEs.
	
	In this paper, we consider the linear Kolmogorov PDE, which plays an important role in finance and physics. Beck et al.~\cite{Beck:2018} first proposed a deep learning algorithm to numerically approximate the solution over a full hypercube and performed experiments to demonstrate the effectiveness of the algorithm even in high dimensions. Berner et al.~\cite{Berner:2020b} further developed the deep learning algorithm for solving a parametric family of high-dimensional Kolmogorov PDEs. Berner et al.~\cite{Berner:2020a} and Jentzen et al.~\cite{Jentzen:2021} provided the theoretical results that the deep learning algorithm based on deep artificial neural networks can overcome the curse of dimensionality under some regularity conditions. Richter et al.~\cite{Richter2022} introduced different loss functions to enhance the robustness of the deep learning algorithm for solving Kolmogorov PDEs.
	
	All of these deep learning-based algorithms reformulate the approximation problem into an optimization problem, where the loss function is defined as a stochastic differential equation~(SDE)-based expectation. Researchers usually employ Monte Carlo~(MC) methods to compute the loss function. Quasi-Monte Carlo~(QMC) methods are more efficient numerical integration methods than MC methods~\cite{Niederreiter:1992}. QMC methods choose deterministic points, rather than random points, as sample points. Due to their effectiveness for integration problems, QMC methods are widely used in finance~\cite{Ecuyer:2009}, statistics~\cite{Fang:2002} and other problems. The applications of QMC methods in deep learning have also made some progress. Dick and Feischl~\cite{Dick:2021} applied QMC methods to compress data in machine learning. Liu and Owen~\cite{Liu:2021} introduced randomized QMC~(RQMC) methods in stochastic gradient descent method to accelerate the convergence. Longo et al.~\cite{Longo:2021higher} and Mishra et al.~\cite{Mishra:2021} trained the neural networks by QMC methods and proved that the QMC-based deep learning algorithm is more efficient than the MC-based one for the evaluation of data-to-observables maps.
	
	In this paper, we combine RQMC methods with deep learning for solving linear Kolmogorov PDEs, leading to the RQMC-based deep learning algorithm. To demonstrate the superiority of our proposed algorithm, we analyze how the error from the empirical risk minimization~(ERM) depends on the sampling methods used to simulate training data. Similar to bias-variance decomposition, we decompose the error into the approximation error and the generalization error. The approximation error is independent of the training data. We obtain the convergence rate of mean generalization error with respect to the sample points for both MC methods and RQMC methods. We also conduct numerical experiments to compare the performance of the RQMC-based deep learning algorithm and the MC-based deep learning algorithm for solving specific linear Kolmogorov PDEs.  
	
	This paper is organised as follows. In Section \ref{sec:pre}, we introduce the minimization problem associated with the linear Kolmogorov PDEs and generalize it into a general deep learning Framework \ref{frame:lpwui}, we also introduce preliminary knowledge of the RQMC method and bias-variance type decomposition of the estimation error. In Section \ref{sec:erranalysis}, we obtain the convergence rate of the mean generalization error for different sampling methods and for specific linear Kolmogorov PDEs with affine drift and diffusion. In Section \ref{sec:numeriExp}, we implement the RQMC-based deep learning algorithm and the MC-based deep learning algorithm to solve the Black-Scholes PDEs and heat equations. Section~\ref{sec:conclu} concludes the paper.
	
	\section{Preliminaries}\label{sec:pre}
	\subsection{General deep learning framework for solving linear Kolmogorov PDEs}\label{sec:2.1}
	We consider the linear Kolmogorov PDE
	\begin{equation}
		\label{eq:target}
		\begin{cases}
			(\nabla_{t}u)(t,x) = \frac{1}{2}\mathrm{Trace}\left( \sigma(x)[\sigma(x)]^{*}(\nabla_{x}^{2}u)(t,x)\right) +\seq{\mu(x),(\nabla_{x}u)(t,x)}_{\mathbb{R}^{d}},	\\
			u(0,x) = \varphi(x),\\
		\end{cases}
	\end{equation}
	where $\varphi (x) \in C(\mathbb{R}^{d},\mathbb{R}) $, $\mu(x) \in C(\mathbb{R}^{d},\mathbb{R}^{d}) $ and $\sigma(x) \in C(\mathbb{R}^{d},\mathbb{R}^{d\times d})$. Let $T\in (0,\infty)$ and $u^*(t,x) \in C^{1,2}([0,T]\times\mathbb{R}^{d},\mathbb{R})$ be the solution of~\eqref{eq:target}. The goal is to numerically approximate the endpoint solution $u^{*}(T,\cdot)$ on the hypercube $[a,b]^d$. Beck et al.~\cite{Beck:2018} prove that the target $\left( u^{*}(T,x)\right)_{x\in[a,b]^d} $ solves a minimization problem stated rigorously in Lemma \ref{lemma:vfe}.
	
	\begin{lem}\label{lem:beck}
		\label{lemma:vfe} Let $d \in \mathbb{N}$, $T,L \in (0,\infty)$, $a\in \mathbb{R}$, $b \in (a,\infty)$, let $\mu(x)$ and $\sigma(x)$ in \eqref{eq:target} satisfy for every $x, y \in \mathbb{R}^d$ that 
		$$\norm{\mu(x)-\mu(y)}_{2}+\norm{\sigma(x)-\sigma(y)}_{HS}\leq L\norm{x-y}_{2},$$
		where $\norm{\cdot}_{2}$ is the Euclidean norm and $\norm{\cdot}_{HS}$ is the Hilbert-Schmidt norm. Let the function $u^*(t,x) \in C^{1,2}([0,T]\times\mathbb{R}^{d},\mathbb{R})$ be the solution of~\eqref{eq:target} with at most polynomially growing partial derivatives. Let $(\Omega,\mathcal{F},P)$ be a probability space with a normal filtration $\left( \mathcal{F}_{t}\right) _{t\in [0,T]}$, $\set{B_{t},\mathcal{F}_{t}\mid 0\leq t\leq T}$ be the standard $d$-dimensional Brownian motion, $X: \Omega \to [a,b]^{d}$ be uniformly distributed and $\mathcal{F}_{0}$-measurable. Let $\set{S_{t},\mathcal{F}_{t}\mid 0\leq t\leq T}$ be the continuous stochastic process satisfying that for every $t\in [0,T]$ it holds $\mathbb{P}$-a.s.
		\begin{equation}
			\label{eq:sdeX}
			S_{t} = X + \int_{0}^{t}\mu(S_{u})du + \int_{0}^{t}\sigma(S_{u})dB_{u}.
		\end{equation}
		For every $x\in[a,b]^{d}$, let $\set{S_{t}^{x},\mathcal{F}_{t}\mid 0\leq t\leq T}$ be the continuous stochastic process satisfying that for every $t\in [0,T]$ it holds $\mathbb{P}$-a.s.
		\begin{equation}
			\label{eq:xsde}
			S_{t}^{x} = x + \int_{0}^{t}\mu(S_{u}^{x})du + \int_{0}^{t}\sigma(S_{u}^{x})dB_{u}.	
		\end{equation}
		We define the loss function $\mathcal{L}(f) = \E{\left( f(X)-\varphi(S_{T})\right) ^{2}}$ for every $f \in C([a,b]^{d},\mathbb{R})$. 
		
		Then there exists a unique function $U \in C([a,b]^{d},\mathbb{R})$ such that 
		\begin{equation}
			U = \underset{f \in C([a,b]^{d},\mathbb{R})}{\mathrm{argmin}}\mathcal{L}(f),
			\nonumber
		\end{equation}
		and for every $x\in [a,b]^{d}$ it holds that
		$$U(x) = u^{*}(T,x) = \E{\varphi(S_{T}^{x})}.$$
	\end{lem}
	
	This lemma shows that the numerical approximation problem is equivalent to the supervised learning problem for input $X$ and label $\varphi(S_{T})$ with quadratic loss function. Suppose that $(X_i,S_{T,i})_{i=1}^n$ are independent and identically distributed~(i.i.d.) samples drawn from the population $(X,S_{T})$, then we can employ the empirical risk minimization (ERM) principle to solve this regression problem by minimizing the empirical risk 
	\begin{equation}
		\label{eq:elossdfn}
		\mathcal{L}_n(f) = \frac{1}{n}\sum_{i=1}^n\left[ f(X_{i})-\varphi(S_{T,i})\right] ^{2}
	\end{equation}
	among the hypothesis class $\mathcal{H}\subseteq C([a,b]^{d},\mathbb{R})$.
	
	In deep learning, we use artificial neural networks as the hypothesis class. 
	\begin{defn}[Artificial neural network]\label{dfn:ann}
		Let $L\geq 2$, $N_{0}, N_{1},\dots,N_{L}\in \mathbb{N}, R\in(0,\infty)$, $\rho(x)\in C(\mathbb{R},\mathbb{R})$. The artificial neural network $\Phi_{\theta}:\mathbb{R}^{N_{0}}\to \mathbb{R}^{N_{L}}$ with the activation function $\rho$, the parameter bound $R$ and the structure $S=(N_{0},\dots,N_L)$ is defined by
		\begin{equation}\label{eq:ANN}
			\Phi_{\theta}(\cdot)=W_{L}\circ\rho \circ W_{L-1}\circ\rho\cdots\circ\rho\circ W_{1}(\cdot), \ \ \theta \in \Theta_{R}, 
		\end{equation}
		where $\set{W_{l}(x) = A_{l}x+B_{l}: A_{l}\in\mathbb{R}^{N_{l}\times N_{l-1}}, B_{l} \in \mathbb{R}^{N_{l-1}}, 1\leq l\leq L}$ are affine transformations, $\theta=((A_{l},B_{l}))_{l=1}^{L}$ denotes the parameters in the network, $\Theta_{R}$ is the parameter space defined by $\Theta_{R}=\set{\theta:\norm{\theta}_{\infty}\leq R}$, and the activation function $\rho$ is applied component-wisely. 
		
		Denote the function class of all artificial neural network \eqref{eq:ANN} by $\mathcal{N}_{\rho,R,S}$, and denote the restriction of $\mathcal{N}_{\rho,R,S}$ on $[a,b]^{d}$ by
		$$\mathcal{N}^{a,b}_{\rho,R,S}=\set{\Phi_{\theta}|_{[a,b]^{d}};\theta \in \Theta_{R}}.$$
	\end{defn} 
	\begin{rem}
		The total number of parameters is $P(S)=\sum_{l=1}^{L}(N_{l}N_{l-1}+N_{l})$, $\theta$ is equivalent to vectors on $\mathbb{R}^{P(S)}$, thus $\Theta_{R}$ is well-defined and is a compact set on $\mathbb{R}^{P(S)}$. Since the composition of continuous functions remains continuous, we have $\mathcal{N}^{a,b}_{\rho,R,S}\subseteq C([a,b]^{d},\mathbb{R})$. In this paper, we choose swish activation function 
		$$\rho(x)=\frac{x}{1+\exp(-x)},$$
		which could improve the performance of deep learning compared to RELU and sigmoid functions \cite{Ramachandran:2017,Rasamoelina:2020}.
	\end{rem}
	
	Since artificial neural networks $\Phi_{\theta}$ is completely determined by its parameters $\theta$, it suffices to search the optimal $\theta$ which minimizes the empirical risk over the artificial neural networks
	$$\underset{f\in \mathcal{N}_{\rho,R,S}^{a,b}}{\mathrm{argmin}}\mathcal{L}_{n}(f) \Rightarrow \underset{\theta \in \Theta_{R}}{\mathrm{argmin}}\mathcal{L}_{n}(\Phi_{\theta}).$$
	Deep learning can be applied to solve the linear Kolmogorov PDE \eqref{eq:target} if we can simulate the training data $(X_i,S_{T,i})_{i=1}^n$.   
	
	For important cases of linear kolmogorov PDEs such as the heat equation and the Black-Scholes PDE, the SDE \eqref{eq:sdeX} can be solved explicitly. The closed form of $S_{T}$ depends only on the initial condition $X$ and the Brownian motion $B_{T}$ which can be simulated by
	\begin{equation}
		\label{eq:specialcase}
		X=a+(b-a)U_{1}, \quad B_{T}=\sqrt{T}\Phi^{-1}(U_{2}),
	\end{equation}
	respectively, where $U_{1},U_{2}$ are uniformly distributed on $(0,1)^{d}$, $\Phi(\cdot)$ is the cumulative distribution function of the standard normal and $\Phi^{-1}(\cdot)$ is the inverse applied component-wisely.
	
	In general cases, we usually approximate $S_{T}$ numerically by Euler–Maruyama scheme. Let $M\in \mathbb{N}$, the discrete process $(S_{j}^{M})_{j=0}^{M}$ is defined by the following recursion
	\begin{equation} \label{eq:EMrecur}
		S_{0}^{M}=X \quad \mathrm{and} \quad S_{m+1}^{M}=S_{m}^{M}+\mu\left(S_{m}^{M}\right)\Delta T+\sigma\left(S_{m}^{M}\right) \left( B_{(m+1) \Delta T}-B_{m \Delta T}\right) , 
	\end{equation}
	where $\Delta T = T/M$ and $0\leq m\leq M-1$. We actually have an approximated learning problem with loss function $$\mathcal{L}^{(M)}(f)=\E{\left( f(X)-\varphi\left(S_{M}^{M}\right) \right) ^{2}}.$$
	Theorem 2 in \cite{Berner:2020b} shows that solving the learning problem with $\mathcal{L}^{(M)}(\cdot)$ does indeed result in a good approximation of the target $u^{*}(T,\cdot)$. 
	
	For the approximated learning problem, since we have for $0\leq m\leq M-1$,
	\begin{equation}
		\label{eq:genecase}
		X=a+(b-a)U_{1},\quad B_{(m+1)\Delta T}-B_{m\Delta T}=\sqrt{\Delta T}\Phi^{-1}(U_{m+2}),
	\end{equation}
	we can simulate $(X,S_{M}^{M})$ by uniformly distributed random variables $U_{1},U_{2},\dots,U_{M+1}$.
	
	By summarizing the \eqref{eq:specialcase} and \eqref{eq:genecase}, we formulate a general framework.
	
	\begin{framework}[Deep learning problem with uniform input]\label{frame:lpwui}
		Let $a\in \mathbb{R}$, $b\in (a,\infty)$, $d,D\in \mathbb{N}$, let $U_{1}$ be uniformly distributed on $(0,1)^{d}$ and $U_{2}$ be uniformly distributed on $(0,1)^{D}$, denote $U = (U_{1},U_{2})$. Assume that the input $X$ and output $Y$ are given by
		\begin{equation}
			\label{eq:xy}
			X = a+(b-a)U_{1} \quad \mathrm{and} \quad Y = F(U),
		\end{equation}
		where $F\in L^{2}([0,1]^{d+D},\mathbb{R}) $. Let $\set{\bm u_{i}=(\bm u_{i,1},\bm u_{i,2}), i\in \mathbb{N}}$ be a sequence of random variables satisfying $\bm u_{i}\sim U$, define the corresponding input data $\bm x_{i}$ and output data $\bm y_{i}$ by
		\begin{equation}
			\bm x_{i} = a+(b-a)\bm u_{i,1} \quad \text{and} \quad \bm y_{i} = F(\bm u_{i}).
		\end{equation}
		Let $n \in \mathbb{N}$, for every $f\in L^{2}([a,b]^{d},\mathbb{R})$, define the risk $\mathcal{L}$ and the empirical risk $\mathcal{L}_{n}$ by 
		\begin{equation}
			\label{eq:dfnloss}
			\mathcal{L}(f) = \E{\left( f(X)-Y\right) ^{2}} \quad \text{and} \quad \mathcal{L}_{n}(f) = \frac{1}{n}\sum_{i=1}^{n}\left[  f(\bm x_{i})-\bm y_{i}\right] ^{2}.
		\end{equation}
		Let $C([a,b]^{d},\mathbb{R})$ be the Banach space of continuous functions with norm $\norm{\cdot}_{\infty}$ defined by 
		$$\norm{f}_{\infty} = \underset{x\in [a,b]^{d}}{\max}\abs{f(x)}.$$ 
		Define the optimizer of risk $\mathcal{L}$ on $C([a,b]^{d},\mathbb{R})$ by
		$$\quad f^{*} = \underset{f\in C([a,b]^{d},\mathbb{R})}{\mathrm{argmin}}\mathcal{L}(f).$$
		Let $L\geq 2, S = \left(d, N_{1},\dots,N_{L-1},1\right)\in \mathbb{N}^{L+1} , R\in(0,\infty)$,  
		\begin{equation}\label{eq:swish}
			\rho(x)=\frac{x}{1+\exp(-x)},
		\end{equation}
		assume the hypothesis class $\mathcal{H}=\mathcal{N}^{a,b}_{\rho,R,S}$ (see Definition \ref{dfn:ann}). Define approximations $f_{\mathcal{H}}$ and $f_{n,\mathcal{H}}$ by 
		\begin{equation}
			f_{\mathcal{H}} = \underset{f\in \mathcal{H}}{\mathrm{argmin}}\mathcal{L}(f) \quad \mathrm{and} \quad f_{n,\mathcal{H}} = \underset{f\in \mathcal{H}}{\mathrm{argmin}}\mathcal{L}_{n}.
		\end{equation}
	\end{framework}
	
	The deep learning problem in Framework \ref{frame:lpwui} considers the input $X$ and output $Y$ with specific dependence \eqref{eq:xy} on the uniform random variable $U$. In classical learning theory \cite{Anthony:1999,Cucker:2007,Shalev:2014}, the output is usually assumed to be bounded. However, in many applications the output $Y$ in \eqref{eq:xy} does not necessarily to be bounded, which means that most theoretical results in learning theory cannot be applied directly for error analysis in Section \ref{sec:erranalysis}.
	
	The classical ERM principle requires i.i.d. data $(\bm u_{i})_{i=1}^{\infty}$ satisfying $\bm u_{i}\sim U$ to compute the empirical risk $\mathcal{L}_{n}(\cdot)$. In this scenario, empirical risk is actually the crude MC estimator of the risk. From the theory of QMC and RQMC methods, there are other simulation methods of the uniform data on the unit cube, thus the uniform sequence $(\bm u_{i})_{i=1}^{\infty}$ in Framework \ref{frame:lpwui} does not have to be independent. If $(\bm u_{i})_{i=1}^{\infty}$ are chosen to be the scrambled digital sequence \cite{Owen:1995}, then the empirical risk $\mathcal{L}_{n}(f)$ is the RQMC estimator of the risk $\mathcal{L}(f)$. 
	
	\subsection{Randomized quasi-Monte Carlo methods}
	Consider the approximation of an integral
	$$I = \int_{[0,1]^d}g(u)du.$$
	MC and QMC estimators both have the form
	$$I_{n} = \frac{1}{n}\sum_{i=1}^{n}g(\bm u_{i}).$$
	MC methods use i.i.d. uniformly distributed sample from $[0,1]^{d}$. Suppose $g(\bm u_{i})$ has finite variance $\sigma^{2}$, then we have 
	$$\E{\abs{I_{n}-I}}\leq \sqrt{\E{\left( I_{n}-I\right)^{2} }}=\sigma n^{-\frac{1}{2}}.$$
	Thus for MC estimator, the convergence rate of the expected error is $O(n^{-1/2})$. For QMC methods, the points $(\bm u_{i})_{i=1}^{n}$ are deterministic, which are highly uniformly distributed on $[0,1]^{d}$. The QMC error bound is given by the Koksma-Hlawka inequality \cite{Niederreiter:1992}
	\begin{equation}\label{eq:koksma}
		\abs{\frac{1}{n}\sum_{i=1}^{n}g(\bm u_{i})-\int_{[0,1]^d}g(u)du}\leq V_{HK}(g)D^{*}(\bm u_{1},\bm u_{2},\dots,\bm u_{n}),
	\end{equation}
	where $V_{HK}(g)$ is the variation of $g$ in the sense of Hardy and Krause, $D^{*}(u_{1},\dots,u_{n})$ is the star discrepancy of the point set $(\bm u_{i})_{i=1}^{d}$. Several digital sequences such as the Sobol' sequence and the Halton sequence are designed to have low star discrepancy with $D^{*}(\bm u_{1},\dots,\bm u_{n}) = O(n^{-1}(\log n)^{d})$. We refer to \cite{Niederreiter:digital,Niederreiter:1992} for the detailed constructions. For digital sequence $(\bm u_{i})_{i=1}^{\infty}$, suppose that $V_{HK}(g)<\infty$, Koksma-Hlawka inequality shows that the convergence rate of the approximation error is $ O(n^{-1}(\log n)^{d})$ which is asymptotically better than $O(n^{-1/2})$ of MC.
	
	Since QMC points are deterministic, the corresponding estimators are not unbiased. We can use RQMC methods which randomize each point to be uniformly distributed while preserving the low discrepancy property, see \cite{Cranley:1976,Owen:1995} for various randomization methods. 
	
	In this paper, we consider the scrambled digital sequence introduced in \cite{Owen:1995}. Let $(\bm u_{i})_{i=1}^{\infty}$ be the scrambled $(t,d)$-sequence in base b, then we have the following properties:
	\begin{enumerate}
		\item[(i)] Every $\bm u_{i}$ is uniformly distributed on $(0,1)^{d}$.
		\item[(ii)]The sequence $(\bm u_{i})_{i=1}^{\infty}$ is a $(t,d)$-sequence in base $b$ with probability 1.
		\item[(iii)] There exists a constant $B$ such that for every $m\in \mathbb{N}$ and $n=b^{m}$ it holds that
		\begin{equation}\label{eq:Koksmabound}
			\mathbb{P}\left\lbrace D^{*}(\bm u_{1},\dots,\bm u_{n})\leq B\frac{(\log n)^{d}}{n}\right\rbrace = 1 .
		\end{equation}
	\end{enumerate}
	We refer the readers to \cite{Dick:2010,Niederreiter:1992,Owen:1995} for details of the $(t,d)$-sequence and the scrambling methods.
	\subsection{Analysis of the estimation error}
	Under the Framework \ref{frame:lpwui}, we are interested in the estimation error $\mathcal{E}(f_{n,\mathcal{H}})$ defined by
	$$\mathcal{E}\left( f_{n,\mathcal{H}}\right) =\mathbb{E}\left[ \left( f_{n,\mathcal{H}}(X)-f^{*}(X)\right) ^{2}\right]= \frac{1}{(b-a)^{d}}\norm{f_{n,\mathcal{H}}-f^{*}}_{L^{2}([a,b]^{d})}^{2}.$$
	The following lemma shows that $\mathcal{E}(f_{n,\mathcal{H}})$ can be expressed in terms of the loss function.
	\begin{lem}\label{lem:excessrisk}
		Under Framework \ref{frame:lpwui}, for every $f\in C([a,b]^{d},\mathbb{R})$, we have
		\begin{equation}
			\E{\left( f(X)-f^{*}(X)\right) ^{2}}=\mathcal{L}(f)-\mathcal{L}(f^{*}). \nonumber
		\end{equation}
	\end{lem}
	\begin{proof}
		Consider $q(t)=\mathcal{L}(tf+(1-t)f^{*})$, $q(t)$ is a quadratic function and takes the minimum at $t=0$, thus $q(1)-q(0)$ equal to the coefficient of $t^{2}$ which is exactly $\E{\left( f(X)-f^{*}(X)\right) ^{2}}$. Thus we have
		$$	\mathcal{L}(f)-\mathcal{L}(f^{*}) =q(1)-q(0)=\E{\left( f(X)-f^{*}(X)\right) ^{2}},$$
		the proof completing. $\square$
	\end{proof}
	
	In classical learning theory, there are many theoretical results which establish conditions on the sample size $n$ and the hypothesis class $\mathcal{H}$ in order to obtain an error 
	$$\mathcal{L}(f_{n,\mathcal{H}})-\mathcal{L}(f^{*})\leq \epsilon$$
	with high probability for small $\epsilon>0$, see \cite{Anthony:1999,Berner:2020a,Cucker:2002,Koltchinskii:2011,Shalev:2014}. Those theoretical results rely on the i.i.d. assumptions of the training data and the boundness condition of the output, thus they cannot be applied directly to research on the influence of sampling methods of $(\bm u_{i})_{i=1}^{\infty}$. Without the i.i.d. assumption, we still have the following decomposition
	\begin{eqnarray}\label{eq:errrorComposition}
		\mathcal{L}(f_{n,\mathcal{H}})-\mathcal{L}(f^{*})&=& \mathcal{L}(f_{n,\mathcal{H}})-\mathcal{L}(f_{\mathcal{H}})+\mathcal{L}(f_{\mathcal{H}})-\mathcal{L}(f^{*})\\
		&=&\underbrace{\mathcal{L}(f_{n,\mathcal{H}})-\mathcal{L}(f_{\mathcal{H}})}_{\text{generalization error}}+\underbrace{(b-a)^{-d}\norm{f_{\mathcal{H}}-f^{*}}_{L^{2}([a,b]^{d})}^{2}}_{\text{approximation error}}. \nonumber
	\end{eqnarray}
	
	In consistent with \cite{Beck_2022,Berner:2020a}, we refer the term $\mathcal{L}(f_{n,\mathcal{H}})-\mathcal{L}(f_{\mathcal{H}})$ as the generalization error and $(b-a)^{-d}\norm{f_{\mathcal{H}}-f^{*}}_{L^{2}([a,b]^{d})}^{2}$ as the approximation error. The approximation error depends only on the function class $\mathcal{H}=\mathcal{N}^{a,b}_{\rho,R,S}$, see \cite{Burger:2001error,Guhring:2021approximation,Hornik:1991approximation,Ohn:2019smooth} for approximation property of artificial neural networks. To address the influence of sampling methods, we keep the artificial neural networks fixed and focus on the generalization error. The next lemma establishes a upper bound for the generalization error.
	\begin{lem}\label{lem:upperbound}
		Under Framework \ref{frame:lpwui}, we have
		\begin{eqnarray*}
			0\leq \mathcal{L}(f_{n,\mathcal{H}})-\mathcal{L}(f_{\mathcal{H}})&\leq& \underset{f\in \mathcal{H}}{\mathrm{sup}} \left[\mathcal{L}(f)-\mathcal{L}_{n}(f)\right]+\underset{f\in \mathcal{H}}{\mathrm{sup}} \left[\mathcal{L}_{n}(f)-\mathcal{L}(f)\right]  \\
			&\leq&2\underset{f\in \mathcal{H}}{\mathrm{sup}}\abs{ \mathcal{L}(f)-\mathcal{L}_{n}(f)}
		\end{eqnarray*}
		
	\end{lem}
	\begin{proof}
		By $f_{\mathcal{H}} = \mathrm{argmin}_{f\in \mathcal{H}}\mathcal{L}(f)$, it is trivial that
		$$\mathcal{L}(f_{\mathcal{H}})\leq\mathcal{L}(f_{n,\mathcal{H}}).$$
		For the upper bounds, we have
		\begin{eqnarray*}
			\mathcal{L}(f_{n,\mathcal{H}})-\mathcal{L}(f_{\mathcal{H}})&=& \mathcal{L}(f_{n,\mathcal{H}})-\mathcal{L}_{n}(f_{n,\mathcal{H}}) + \mathcal{L}_{n}(f_{n,\mathcal{H}})-\mathcal{L}_{n}(f_{\mathcal{H}})+\mathcal{L}_{n}(f_{\mathcal{H}})-\mathcal{L}(f_{\mathcal{H}})\\
			&\leq&\mathcal{L}(f_{n,\mathcal{H}})-\mathcal{L}_{n}(f_{n,\mathcal{H}}) +\mathcal{L}_{n}(f_{\mathcal{H}})-\mathcal{L}(f_{\mathcal{H}})\\
			&\leq&\underset{f\in \mathcal{H}}{\mathrm{sup}} \left[\mathcal{L}(f)-\mathcal{L}_{n}(f)\right]+\underset{f\in \mathcal{H}}{\mathrm{sup}} \left[\mathcal{L}_{n}(f)-\mathcal{L}(f)\right]\\
			&\leq&2\underset{f\in \mathcal{H}}{\mathrm{sup}}\abs{ \mathcal{L}(f)-\mathcal{L}_{n}(f)},
		\end{eqnarray*}
		where the second inequality follows from $\mathcal{L}_{n}(f_{n,\mathcal{H}})\leq\mathcal{L}_{n}(f_{\mathcal{H}})$. $\square$
	\end{proof}

	\section{Main Results}\label{sec:erranalysis}
	In this section, we obtain the convergence rate with respect to the sample size $n$ of the mean generalization error for MC and RQMC methods. The mean generalization error is defined by 
	\begin{equation}
		\label{eq:expectedSampleError}
		\mathbb{E}\left[\mathcal{L}(f_{n,\mathcal{H}})-\mathcal{L}(f_{\mathcal{H}})\right].
	\end{equation}
	Throughout this section, we keep the hypothesis class $\mathcal{H}=\mathcal{N}^{a,b}_{\rho,R,S}$ fixed.
	\subsection{RQMC-based convergence rate}
	By Lemma \ref{lem:upperbound}, it suffices to obtain the convergence rate of the upper bound
	\begin{equation}
		\label{eq:RqmcUpper}
		\mathbb{E}\left[\underset{f\in \mathcal{H}}{\mathrm{sup}}\abs{ \mathcal{L}(f)-\mathcal{L}_{n}(f)}\right] .
	\end{equation}
	Assume that we choose the uniform sequence $(\bm u_{i})_{i=1}^{\infty}$ in Framework \ref{frame:lpwui} to be the scrambled digital sequence, then \eqref{eq:RqmcUpper} equals to the mean supreme error of scrambled net quadrature rule that is 
	$$\E{\underset{\theta \in \Theta_{R}}{\mathrm{sup}}\abs{\frac{1}{n}\sum_{i=1}^n g_{\theta}( \bm u_i)-\int_{[0,1]^d}g_{\theta}( u) d  u}},$$
	the function $g_{\theta}(u)$ is defined by 
	\begin{equation}
		\label{eq:paraClass}
		g_{\theta}(u)=\left[\Phi_{\theta}\left( a+(b-a)u_{1:d}\right) -F(u) \right]^{2},
	\end{equation}
	where $u_{1:d}$ denotes the first $d$ components of $u$ and $\Phi_{\theta} \in \mathcal{H}$ is the artificial neural network.
	In most cases, the functions $g_{\theta}(u)$ are unbounded and cannot be of bounded variation in the sense of Hardy and Krause. For unbounded functions satisfying the boundary growth condition, we can still obtain the convergence rate of the mean error (or even root mean squared error), see \cite[Theorem 3.1]{He:2023error} and \cite[Theorem 5.7]{Owen:2006}.  These results cannot be applied directly since we need to handle with the supreme error for a function class, but it is natural to introduce the boundary growth condition for a whole function class.
	
	For a nonempty set $v \subseteq \set{1,\cdots,d}$, denote the partial derivative $(\prod_{i\in v}\partial/\partial x_i)g(u)$ by $D_{v}g(u)$, we make a convention that $D_{\emptyset}g(u) = g(u)$. Let $\bm 1\set{\cdot}$ be an indicator function.
	
	\begin{defn}[Boundary growth condition]
		Let $d\in \mathbb{N}$, suppose $\mathcal{G}$ is a class of real-valued functions defined on $(0,1)^d$. We say that $\mathcal{G}$ satisfies the boundary growth condition with constants $(B_i)_{i=1}^{d}$ if there exists $B\in (0,\infty)$ such that for every $g \in \mathcal{G}$, every subset $v \subseteq \set{1,\cdots,d}$ and every $u=(u_{1},\dots,u_{d}) \in (0,1)^{d}$ it holds that
		\begin{equation}\label{eq:gc}
			\abs{(D_vg)( u)}\leq B\prod_{i=1}^d[\min(u_i,1-u_i)]^{-B_i-\bm 1\{i\in v\}}.
		\end{equation}
	\end{defn}
	
	The next theorem establishes the error rate of mean supreme error for a function class satisfying the boundary growth condition \eqref{eq:gc}.  
	\begin{thm}\label{thm:errqmc}
			Suppose $\mathcal{G}$ is a class of real-valued functions defined on $(0,1)^d$ which satisfies the boundary growth condition \eqref{eq:gc} with constants $(B_i)_{i=1}^{d}$. Suppose that $(\bm u_{i})_{i=1}^{\infty}$ is a scrambled $(t,d)$-sequence in base $b\geq 2$, let sample size $n=b^{m}$, $m\in \mathbb{N}$, then we have
			\begin{equation}
				\E{\underset{g\in \mathcal{G}}{\mathrm{sup}}\abs{\frac{1}{n}\sum_{i=1}^n g( \bm u_i)-\int_{[0,1]^d}g( u) d  u}}=O\left( n^{-1+\max_{i}B_i}(\log n)^{d}\right).
			\end{equation}
	\end{thm}
	\begin{proof}
		To prove this theorem, we need to introduce concept of the low variation extension \cite{Owen:2005,Owen:2006}. Detailed definitions and useful properties can be found in \ref{appendix: lve}. For $g\in \mathcal{G}$ and $\eta \in (0,1/2)$, let $g_{\eta}$ be the low variation extension of $g$, denote $I(g) = \int_{(0,1)^d}g( u) d  u$ and $\hat{I}(g)=\frac{1}{n}\sum_{i=1}^n g(\bm u_i)$. By the triangle inequality, we have
		\begin{eqnarray*}
			\E{\underset{g\in \mathcal{G}}{\mathrm{sup}}\abs{\hat{I}(g)-I(g)}}&\leq& \E{\underset{g\in \mathcal{G}}{\mathrm{sup}}\abs{\hat{I}(g-g_{\eta})}}+\underset{g\in \mathcal{G}}{\mathrm{sup}}\abs{I(g-g_{\eta})}\\
			&&+\E{\underset{g\in \mathcal{G}}{\mathrm{sup}}\abs{\hat{I}(g_{\eta})-I(g_{\eta})}} \\
			&\leq&2C_2\eta^{1-\max_{i}B_i}+\E{\underset{g\in \mathcal{G}}{\mathrm{sup}}\abs{\hat{I}(g_{\eta})-I(g_{\eta})}} ,
		\end{eqnarray*}
		the second inequality follows from $(i)$ of Lemma \ref{lem:lveproperty}. Applying the Koksma-Hlawka inequality \eqref{eq:koksma}, $(iii)$ of Lemma \ref{lem:lveBasicproperty} and the formula $\eqref{eq:Koksmabound}$, we find that 
		\begin{eqnarray*}
			\E{\underset{g\in \mathcal{G}}{\mathrm{sup}}\abs{\hat{I}(g_{\eta})-I(g_{\eta})}}&\leq&\underset{g\in \mathcal{G}}{\mathrm{sup}}V_{\mathrm{HK}}\left( g_{\eta}\right)\E{D^{*}(\bm u_{1},\dots, \bm u_{n})}  \\
			&\leq&\underset{g\in \mathcal{G}}{\mathrm{sup}}V_{\mathrm{HK}}\left( g_{\eta}\right)\times O\left( \frac{(\log n)^{d}}{n}\right)  \\
			&=&O\left( \eta^{-\max_{i}B_i}\right) \times  O\left( \frac{(\log n)^{d}}{n}\right) .
		\end{eqnarray*}
		We can conclude that 
		\begin{equation}
			\E{\underset{g\in \mathcal{G}}{\mathrm{sup}}\abs{\hat{I}(g)-I(g)}} \leq 	O\left( \eta^{1-\max_{i}B_i}\right) +O\left( \eta^{-\max_{i}B_i}\right) \times  O\left( \frac{(\log n)^{d}}{n}\right), \nonumber
		\end{equation}
		taking $\eta \propto n^{-1}$, the upper bound become $O\left( n^{-1+\max_{i}B_i}(\log n)^{d}\right) $. $\square$
	\end{proof}
	
	If the function class satisfies boundary growth condition \eqref{eq:gc} with arbitrarily small $B_i>0$, then the error rate in Theorem \ref{thm:errqmc} becomes $O(n^{-1+\epsilon})$ for arbitrarily small $\epsilon > 0$. To obtain the error rate for the expected generalization error \eqref{eq:expectedSampleError}, we only need to verify that the boundary growth condition is satisfied for the specific function class $\set{g_{\theta}(\cdot)\mid \theta \in \Theta_{R}}$ defined by \eqref{eq:paraClass}. The next lemma provides an easy way to verify the boundary growth condition for a complicated function class.
	
	\begin{lem}\label{lem:addmul}
		Suppose that $\set{g_{\theta}(u);\theta \in \Theta}$ and $\set{h_{\theta}(u);\theta \in \Theta}$ both satisfy the boundary growth condition \eqref{eq:gc} with constants $(B_i)_{i=1}^{d}$. Then  $\set{g_{\theta}(u)+h_{\theta}(u);\theta \in \Theta}$ and $\set{g_{\theta}(u)h_{\theta}(u);\theta \in \Theta}$ also satisfy the boundary growth condition with the same constants $(B_i)_{i=1}^{d}$.
	\end{lem}
	\begin{proof}
		For every subset $v$ of indices and $\theta \in \Theta$, the partial derivatives of $g_{\theta}+h_{\theta}$ and $g_{\theta}h_{\theta}$ are given by
		\begin{eqnarray*}
			D_v\left(g_{\theta}+h_{\theta} \right) &=& D_v\left(g_{\theta} \right)+D_v\left(h_{\theta} \right) \\
			D_v\left(g_{\theta}h_{\theta} \right) &=& \sum_{w \subseteq v}D_w\left(g_{\theta} \right)D_{v-w}\left(h_{\theta} \right). 
		\end{eqnarray*}
		We can use the above formulas to prove that the partial derivatives of $g_{\theta}+h_{\theta}$ and $g_{\theta}h_{\theta}$ also satisfy the inequality \eqref{eq:gc} with constants  $(B_i)_{i=1}^{d}$. $\square$
	\end{proof}
	
	Based on Theorem \ref{thm:errqmc} and Lemma \ref{lem:addmul}, we can obtain the convergence rate of the mean generalization error if the data $(\bm u_{i})_{i=1}^{\infty}$ is the scrambled digital sequence.
	
	\begin{thm}\label{thm:rqmcge}
		Under Framework \ref{frame:lpwui}, suppose that the function $F(\cdot)$ satisfies the boundary growth condition \eqref{eq:gc} with arbitrarily small $(B_i)_{i=1}^{d+D}$. Suppose that $( \bm u_{i})_{i=1}^{\infty}$ is a scrambled $(t,d+D)$-sequence in base $b\geq 2$. Let sample size $n=b^{m}$, $m\in \mathbb{N}$, then for arbitrarily small $\epsilon>0$ we have 
		\begin{equation}
			\mathbb{E}\left[ \mathcal{L}(f_{n,\mathcal{H}})-\mathcal{L}(f_{\mathcal{H}})\right]  = O(n^{-1+\epsilon}).
		\end{equation}
	\end{thm}
	\begin{proof}
		Under Framework \ref{frame:lpwui}, the function class $\mathcal{N}^{a,b}_{\rho,R,S}$ is the restriction of artificial neural networks on $[a,b]^{d}$. To deal with uniform input, we define the function class $$\mathcal{H}_{1}=\set{h_{\theta}:[0,1]^{d}\to \mathbb{R}\mid h_{\theta}(u)=\Phi_{\theta}(a+(b-a)u),\Phi_{\theta}\in \mathcal{N}^{a,b}_{\rho,R,S},\theta\in \Theta_{R}}.$$
		Due to the smoothness of swish activation function $\rho$, $h_{\theta}(u)$ is a smooth function of $(\theta,u)\in \Theta_{R}\times[0,1]^{d}$. The region $\Theta_{R}\times[0,1]^{d}$ is compact, so there exists constant $B>0$ such that
		for every subset $v \subseteq \set{1,\cdots,d}$ it holds that
		\begin{equation*}
			\underset{(\theta,u)\in \Theta_{R}\times[0,1]^{d}}{\mathrm{sup}}\abs{(D_vh_{\theta})(u)}\leq B.
		\end{equation*}
		Hence $\mathcal{H}_{1}$ satisfies the boundary growth condition with arbitrarily small $B_{i}>0$. By Lemma~\ref{lem:addmul} and the assumption on the function $F$, the following function class 
		$$\set{g_{\theta}:[0,1]^{d+D}\to \mathbb{R}\mid g_{\theta}(u)=\left[\Phi_{\theta}(a+(b-a)u_{1:d})-F(u) \right]^{2},f_{\theta}\in \mathcal{N}^{a,b}_{\rho,R,S},\theta\in \Theta_{R}}$$
		also satisfies the boundary growth condition with arbitrarily small $B_{i}>0$.
		By Lemma~\ref{lem:upperbound} and Theorem \ref{thm:errqmc}, we find that for arbitrarily small $\epsilon >0 $, it holds that
		\begin{eqnarray*}
			\mathbb{E}\left[ \mathcal{L}(f_{n,\mathcal{H}})-\mathcal{L}(f_{\mathcal{H}})\right] &\leq&2\mathbb{E}\left[ \underset{f \in \mathcal{N}^{a,b}_{\rho,R,S}}{\mathrm{sup}}\abs{\mathcal{L}_{n}(f)-\mathcal{L}(f)}\right] \\
			&=&2\mathbb{E}\left[ \underset{\theta\in \Theta_{R}}{\mathrm{sup}}\abs{\frac{1}{n}\sum_{i=1}^n g_{\theta}(\bm u_i)-\int_{(0,1)^d}g_{\theta}( u) d  u}\right] \\
			&=&O(n^{-1+\epsilon}).
		\end{eqnarray*}
		The proof is completed. $\square$
	\end{proof}
	
	For the scrambled digital sequence $(\bm u_{i})_{i=1}^{\infty}$, if the function $F(\cdot)$ satisfies the boundary growth condition \eqref{eq:gc} with arbitrarily small $B_i>0$, Theorem \ref{thm:rqmcge} shows that the mean generalization error converges to $0$ as the sample size $n\to \infty$ and the convergence rate is at least $O(n^{-1+\epsilon})$ for arbitrarily small $\epsilon>0$.
	\subsection{MC-based convergence rate}\label{sec:mc-based-convergence-rate}
	
	Under Framework \ref{frame:lpwui}, it is equivalent to directly simulate the i.i.d. data $(\bm x_{i},\bm y_{i})$ if we use MC methods to simulate i.i.d. uniform sequence $(\bm u_{i})_{i=1}^{\infty}$. In the context of nonparametric regression, theoretical results in \cite{bauer2019deep,schmidt2020nonparametric,jiao2023deep,kohler2021rate} demonstrate that when $f^{*}=f_{\mathcal{H}}$ and $Y=F(U)$ is sub-Gaussian or sub-exponential, it holds that
		\begin{equation}
			\label{eq:Nonpara_results}
			\mathbb{E}\left[\mathcal{L}(f_{n,\mathcal{H}})-\mathcal{L}(f_{\mathcal{H}})\right]=O(n^{-1+\varepsilon}). \nonumber
		\end{equation}
		for every $\varepsilon>0$. In general, the target function $f^{*}$ does not necessarily fall within the hypothesis space  $\mathcal{H}$, and the function $F$ may only satisfy the boundary growth condition \eqref{eq:gc}. For the Deep Ritz method in solving elliptic equations, theoretical results in \cite{Yu:2018,Jiao:2021} demonstrate that the mean generalization error satisfies 
		$$\mathbb{E}\left[\mathcal{L}(f_{n,\mathcal{H}})-\mathcal{L}(f_{\mathcal{H}})\right]=O(n^{-1/2+\varepsilon})$$ 
		for every  $\varepsilon>0$. We can prove that the same convergence rate holds for the general regression problem under Framework \ref{frame:lpwui}.

	\begin{thm}\label{thm:mcerrbound}
		Under Framework \ref{frame:lpwui}, suppose that the function $F$ satisfies the boundary growth condition \eqref{eq:gc} with constants $(B_i)_{i=1}^{d+D}$. Suppose that $( \bm u_{i})_{i=1}^{\infty}$ are i.i.d samples simulated from the uniform distribution on $(0,1)^{d+D}$, then we have
		$$
		\mathbb{E}\left[ \mathcal{L}(f_{n,\mathcal{H}})-\mathcal{L}(f_{\mathcal{H}})\right]  = O(n^{-\frac{1}{2}+\sum_{i}B_i}\left( \log n\right) ^{\frac{1}{2}}).
		$$
		Suppose that the constants $(B_i)_{i=1}^{d+D}$ can be chosen arbitrarily small, then for arbitrarily small $\epsilon>0$ we have
		$$
		\mathbb{E}\left[ \mathcal{L}(f_{n,\mathcal{H}})-\mathcal{L}(f_{\mathcal{H}})\right]  = O(n^{-\frac{1}{2}+\epsilon}).
		$$
	\end{thm}   
	\begin{proof}
		We use the Rademacher complexity technique to prove this theorem, we refer the readers to \ref{appendix:rc} for the definition of Rademacher complexity and some useful properties stated in Lemmas \ref{lem:rcbasic}, \ref{lem:rcbound} and \ref{lem:RandANN}. By Lemma \ref{lem:upperbound}, we have
		\begin{equation}\label{eq:part0}
			\begin{aligned}
				\mathcal{L}(f_{n,\mathcal{H}})-\mathcal{L}(f_{\mathcal{H}})
				&\leq\underset{f \in \mathcal{H}}{\mathrm{sup}}\bra{\mathcal{L}(f)-\mathcal{L}_{n}(f)}+\underset{f\in \mathcal{H}}{\mathrm{sup}}\bra{\mathcal{L}_{n}(f)-\mathcal{L}(f)} \\
				&= \underset{f \in \mathcal{H}}{\mathrm{sup}}\pm\left\lbrace \frac{1}{n}\sum_{i=1}^{n}f^2(\bm x_{i})-\E{f^2( x)}\right\rbrace  \\
				&+\underset{f \in \mathcal{H}}{\mathrm{sup}}\pm\left\lbrace \frac{1}{n}\sum_{i=1}^{n}f(\bm x_{i}) \bm y_i-\E{f( x) y}\right\rbrace \\
				&\leq 2\underset{f \in \mathcal{H}}{\mathrm{sup}}\abs{\frac{1}{n}\sum_{i=1}^{n}f^2( \bm x_{i})-\E{f^2( x)}} \\
				&+2\underset{f \in \mathcal{H}}{\mathrm{sup}}\abs{\frac{1}{n}\sum_{i=1}^{n}f(\bm x_{i})\bm y_i-\E{f( x) y}}. 
			\end{aligned}
		\end{equation}
		
		From Lemma \ref{lem:rcbasic} and \ref{lem:rc_neural}, we have 
		\begin{equation}\label{eq:part1}
			\begin{aligned}
				\E{\underset{f \in \mathcal{H}}{\mathrm{sup}}\abs{\frac{1}{n}\sum_{i=1}^{n}f^2( \bm x_{i})-\E{f^2( x)}}}
				&\leq R_{n}(\mathcal{H}^{2};( x_{i})_{i=1}^{n})\\
				&=O(n^{-\frac{1}{2}}\left( \log n\right)^{\frac{1}{2}}).
			\end{aligned}
		\end{equation}
		
		Let $\eta \in (0,1/2)$ and let the function $F_{\eta}$ be the low variation extension of $F$ from $K(\eta)$ to $(0,1)^{d+D}$. By the triangle inequality, for every function $f \in \mathcal{H}$,
		\begin{eqnarray*}
			&&\abs{\frac{1}{n}\sum_{i=1}^{n}f(\bm x_{i})\bm y_{i}-\E{f( x) y}}\\ 
			&\leq&  \abs{\frac{1}{n}\sum_{i=1}^{n}f(\bm x_{i})F_{\eta}( \bm u_i)-\E{f( x)F_{\eta}( u)}} +\abs{\frac{1}{n}\sum_{i=1}^{n}f( \bm x_{i})\left( F_{\eta}-F\right) (\bm u_i)}\\
			&& + \E{\abs{f( x)\left( F_{\eta}-F\right)( u)}}.	
		\end{eqnarray*}
		By $(i)$ of Lemma \ref{lem:lveproperty} and formula \eqref{eq:neural_property}, we have
		
		\begin{eqnarray*}
			\mathbb{E}\left[ \underset{f \in \mathcal{H}}{\mathrm{sup}}\abs{\frac{1}{n}\sum_{i=1}^{n}f(\bm x_{i})\left( F_{\eta}-F\right) (\bm u_i)}\right] 
			&\leq& B\times\mathbb{E}\left[ \frac{1}{n}\sum_{i=1}^{n}\abs{\left( F_{\eta}-F\right) ( \bm u_i)}\right]\\
			&=&O(\eta^{1-\max_{i}B_i}), \\
			\E{\abs{f( x)\left( F_{\eta}-F\right)( u)}} &\leq& B\times\E{\abs{\left( F_{\eta}-F\right)( u)}}\\
			&=&O(\eta^{1-\max_{i}B_i}).
		\end{eqnarray*}
		
		By Lemmas \ref{lem:lveproperty}, \ref{lem:rcbasic}, \ref{lem:rcbound} and \ref{lem:rc_neural}, we have
		\begin{eqnarray*}
			\mathbb{E}\left[ \underset{f \in \mathcal{H}}{\mathrm{sup}}\abs{\frac{1}{n}\sum_{i=1}^{n}f( \bm x_{i})F_{\eta}(\bm u_i)-\E{f( x)F_{\eta}( u)}}\right] 
			&\leq& \underset{u \in (0,1)^d}{\mathrm{sup}}\abs{F_{\eta}(u)}R_{n}(\mathcal{H};( \bm x_{i})_{i=1}^{n})\\
			&=&O(\eta^{-\sum_{i=1}^d B_i})\times O(n^{-\frac{1}{2}}\left( \log n\right) ^{\frac{1}{2}}).
		\end{eqnarray*}
		
		Summing up three parts, we have
		\begin{equation}\label{eq:part2}
			\begin{aligned}
				&\mathbb{E}\left[ \underset{f \in \mathcal{H}}{\mathrm{sup}}\abs{\frac{1}{n}\sum_{i=1}^{n}f( \bm x_{i})\bm y_i-\E{f( x) y}}\right] \\
				&\leq O(\eta^{1-\max_{i}B_i})+O(\eta^{-\sum_{i=1}^d B_i})\times O(n^{-\frac{1}{2}}\left( \log n\right) ^{\frac{1}{2}}) .
			\end{aligned}
		\end{equation}
		
		Concluding \eqref{eq:part0}, \eqref{eq:part1} and \eqref{eq:part2}, we have
		\begin{eqnarray*}
			&&\mathbb{E}\left[ \mathcal{L}(f_{n,\mathcal{H}})-\mathcal{L}(f_{\mathcal{H}})\right]\\
			&\leq& R_{n}(\mathcal{H}^{2};( x_{i})_{i=1}^{n})+O(\eta^{1-\max_{i}B_i})+O(\eta^{-\sum_{i=1}^d B_i})\times O(n^{-\frac{1}{2}}\left( \log n\right) ^{\frac{1}{2}}) \\
			&=&O(n^{-\frac{1}{2}}\left( \log n\right) ^{\frac{1}{2}})+O(\eta^{1-\max_{i}B_i})+O(\eta^{-\sum_{i=1}^d B_i})\times O(n^{-\frac{1}{2}}\left( \log n\right) ^{\frac{1}{2}}).
		\end{eqnarray*}
		
		Taking the optimal rate $\eta \propto n^{-1/2}$, the inequality becomes
		$$
		\mathbb{E}\left[ \mathcal{L}(f_{n,\mathcal{H}})-\mathcal{L}(f_{\mathcal{H}})\right]= O(n^{-1/2+\sum_{i}B_i}\left( \log n\right) ^{\frac{1}{2}}).
		$$
		The proof is completed. $\square$
	\end{proof}
	
	Assume that we apply MC methods to simulate  i.i.d. uniform data $\bm u_{i}$ (equivalently i.i.d. data $(\bm x_{i},\bm y_{i})$) in Framework \ref{frame:lpwui}. If the function $F$ satisfies the boundary growth condition with arbitrarily small $(B_i)_{i=1}^{d+D}$, Theorem \ref{thm:mcerrbound} shows that the convergence rate of the mean generalization error is at least $O(n^{-1/2+\epsilon})$ for arbitrarily small $\epsilon>0$.
	\subsection{Applications for linear Kolmogorov PDE with affine drift and diffusion}
	As discussed in Section \ref{sec:2.1}, the numerical approximation problem of solutions to linear Kolmogorov PDEs can be reformulated as the deep learning problem satisfying Framework~\ref{frame:lpwui}. In this section, we apply Theorems \ref{thm:rqmcge} and \ref{thm:mcerrbound} to obtain the convergence rate of the mean generalization error for specific linear Kolmogorov PDEs with affine drift and diffusion. In Assumption \ref{assum:driftDiffusion} below, case $(i)$ and $(ii)$ correspond to heat equations and Black-Scholes PDEs, respectively, which are two most important cases of linear Kolmogorov PDEs with affine drift and diffusion.
	\begin{assum}\label{assum:driftDiffusion}
		For the linear Kolmogorov PDE \eqref{eq:target}, suppose that the initial function $\varphi \in C^{\infty}(\mathbb{R}^{d},\mathbb{R})$ has polynomially growing partial derivatives, for the drift $\mu(x) $ and diffusion $\sigma(x) $, we consider the following three cases
		\begin{enumerate}
			\item[(i)]$\mu(x)\equiv \overline{\mu} $ and $ \sigma(x)\equiv \overline{\sigma}$, where $\overline{\mu} \in \mathbb{R}^{d}$ and $\overline{\sigma} \in \mathbb{R}^{d\times d}$.
			\item[(ii)]	$\mu(x)= \mathrm{diag}(x_{1},x_{2},\dots,x_{d})\overline{\mu}$ and $\sigma(x)= \mathrm{diag}(x_{1},x_{2},\dots,x_{d})\overline{\sigma}$, where  $\overline{\mu} \in \mathbb{R}^{d}$ and $\overline{\sigma} \in \mathbb{R}^{d\times d}$.
			\item[(iii)]$\mu(x)$ and $ \sigma(x)$ are both affine transformations .
		\end{enumerate}
	\end{assum}
	
	For case $(i)$ in Assumption \ref{assum:driftDiffusion}, we can solve the SDE \eqref{eq:sdeX} explicitly
	$$S_{T}=\overline{\mu}T+X+\overline{\sigma}B_{T}.$$
	
	Let $\Phi(\cdot)$ be the cumulative distribution function of the standard normal and $\Phi^{-1}(\cdot)$ is the inverse applied component-wisely, then we can write the function $F:(0,1)^{d+d}\to \mathbb{R}$ in Framework \ref{frame:lpwui} as 
	\begin{equation}\label{eq:Fab1}
		F(u)=\varphi\left( \overline{\mu}T+x+\overline{\sigma}z\right),
	\end{equation}
	where
	\begin{equation}
		x=a+(b-a)u_{1:d} \ \ , \ \ z=\sqrt{T}\Phi^{-1}\left( u_{(d+1):2d}\right) .\nonumber
	\end{equation}
	
	For case $(ii)$, we can also solve the SDE \eqref{eq:sdeX} explicitly
	\begin{eqnarray*}
		S_{T}&=&\left(S_{T,1},\dots,S_{T,i}\right), 
	\end{eqnarray*}
	where $$S_{T,i} = X_{i}\exp\left( \left( \mu_{i}-\frac{1}{2}\norm{\sigma_i}_{2}^{2}\right) T+\seq{\sigma_i,B_{T}}_{\mathbb{R}^{d}}\right),$$
	$X_{i}$ denotes the $i$-th component of $X$, $\mu_i$ denotes $i$-th component of $\overline{\mu}$ and $\sigma_i$ denotes $i$-th row of $\overline{\sigma}$. We can write the function $F:(0,1)^{d+d}\to \mathbb{R}$ in Framework \ref{frame:lpwui} as 
	\begin{equation}\label{eq:Fab2}
		F(u)=\varphi \left(s_{1},s_{2},\dots,s_{d}\right), 
	\end{equation}
	where
	\begin{align*}
		s_{i} = x_{i}& \exp\left( \left( \mu_{i}-\frac{1}{2}\norm{\sigma_i}_{2}^{2}\right) T+\seq{\sigma_i,z}_{\mathbb{R}^{d}}\right), \ \ 1\leq i \leq d, \nonumber \\
		x_{i}=a+&(b-a)u_{i},\ \
		z=\sqrt{T}\Phi^{-1}(u_{(d+1):2d}),\ \ 1\leq i \leq d. \nonumber
	\end{align*}
	
	For case $(iii)$, we consider the Euler–Maruyama approximation defined by the following recursion
	\begin{equation}
		S_{0}^{M}=X \quad and \quad S_{m+1}^{M}=S_{m}^{M}+\mu(S_{m}^{M})\frac{T}{M}+\sigma(S_{m}^{M})\left( B_{\frac{(m+1)T}{M}}-B_{\frac{mT}{M}}\right) . \nonumber
	\end{equation}
	For affine transformations $\mu$ and $ \sigma$, there exist polynomials $P_{1},\dots,P_{d}:\mathbb{R}^{d+Md}\to \mathbb{R}$ such that 
	\begin{equation}
		S_{M}^{M}=\left( P_{1}\left( \Delta_{0},\Delta_{1},\dots,\Delta_{M}\right),\dots,P_{1}\left( \Delta_{0},\Delta_{1},\dots,\Delta_{M}\right)\right), \nonumber
	\end{equation}
	where
	\begin{equation}
		\Delta_{0} = X, \ \  \Delta_{i} = B_{\frac{iT}{M}}-B_{\frac{(i-1)T}{M}}, \ \ 1\leq i \leq M. \nonumber \\
	\end{equation}
	Then we can write the function $F:(0,1)^{d+Md}\to \mathbb{R}$ in Framework \ref{frame:lpwui} as 
	\begin{equation}\label{eq:Fab3}
		F(u)=\varphi\left(P_{1} \left( x,z_{1},z_{2},\dots,z_{M}\right) ,\dots, P_{d} \left( x,z_{1},z_{2},\dots,z_{M}\right)\right) ,
	\end{equation}
	where
	\begin{equation}
		x=a+(b-a)u_{1:d} \ \ , \ \ z_{i}=\sqrt{\frac{T}{M}}\Phi^{-1}(u_{(id+1):(id+d)}),\ \ 1\leq i \leq M. \nonumber
	\end{equation}

		\begin{lem}\label{lem:checkBGC}
			Functions $F(\cdot)$ given by \eqref{eq:Fab1}-\eqref{eq:Fab3} both satisfy the boundary growth condition \eqref{eq:gc} with arbitrarily small $B_i>0$. 
		\end{lem}
	
	\begin{proof}
		For the function $F$ given in \eqref{eq:Fab1}, since $\varphi$ has polynomially growing partial derivatives, by chain rule, there exists constant $B>0$ and $r>0$ such that for every subset $v\subseteq \set{1,2,\dots,2d}$ it holds that
		\begin{equation}
			\label{eq:Fabupper1}
			\abs{(D_vF)( u)}\leq B\left( 1+\sum_{i=1}^{d}u_{i}^{r}+\sum_{i=d+1}^{2d}\Phi^{-r}(u_{i})\right) \prod_{\substack{i\in v\\d+1\leq i  }}\frac{\partial \Phi^{-1}(u_{i})}{\partial u_{i}}.
		\end{equation}
		For every $x\in (0,1)$, we have 
		\begin{equation}
			\label{eq:PhiProperty}
			\Phi^{-1}(x)=O\left( \sqrt{\log(\min(x,1-x))}\right) \ \ \mathrm{and} \ \ \frac{\partial \Phi^{-1}(x)}{\partial x}=O\left( \left[ \min(x,1-x)\right]^{-1} \right). 
		\end{equation}
		Summarizing \eqref{eq:Fabupper1} and \eqref{eq:PhiProperty}, we have for arbitrarily small $B_{i}>0$,
		$$\sup\limits_{v\subseteq \set{1,2,\dots,2d}}\abs{(D_vF)( u)}=O\left( \prod_{i=1}^{2d}[\min(u_i,1-u_i)]^{-B_i-1\{i\in v\}}\right) ,$$
		which shows that the function $F$ satisfies the boundary growth condition with arbitrarily small $B_i>0$ .
		
		For the function $F$ given in \eqref{eq:Fab2}, similarly there exists constant $B>0$ such that for every subset $v\subseteq \set{1,2,\dots,2d}$ it holds that
		\begin{equation}
			\label{eq:Fabupper2}
			\abs{(D_vF)( u)}\leq \exp\left( B+B\sum_{i=d+1}^{2d}\Phi^{-1}(u_{i})\right) \prod_{\substack{i\in v\\d+1\leq i  }}\frac{\partial \Phi^{-1}(u_{i})}{\partial u_{i}}.
		\end{equation}
		From \eqref{eq:Fabupper2} and \eqref{eq:PhiProperty}, for arbitrarily small $B_{i}>0$, we have 
		$$\sup\limits_{v\subseteq \set{1,2,\dots,2d}}\abs{(D_vF)( u)}=O\left( \prod_{i=1}^{2d}[\min(u_i,1-u_i)]^{-B_i-1\{i\in v\}}\right) .$$
		
		For the function $F$ given in \eqref{eq:Fab3}, by chain rule and the assumption on $\varphi$, there exists constant $B>0$ and $r>0$ such that for every subset $v\subseteq \set{1,2,\dots,Md+d}$ it holds that
		\begin{equation*}
			\label{eq:Fabupper3}
			\abs{(D_vF)( u)}\leq B\left( 1+\sum_{i=1}^{d}u_{i}^{r}+\sum_{i=d+1}^{Md+d}\Phi^{-r}(u_{i})\right) \prod_{\substack{i\in v\\d+1\leq i  }}\frac{\partial \Phi^{-1}(u_{i})}{\partial u_{i}}.
		\end{equation*}
		From \eqref{eq:PhiProperty}, we can prove that $F$ satisfies the boundary growth condition with arbitrarily small $B_i>0$. $\square$
	\end{proof}
	
	Lemma \ref{lem:checkBGC} shows that for linear Kolmogorov PDEs satisfying Assumption \ref{assum:driftDiffusion}, the corresponding functions $F(\cdot)$ in Framework \ref{frame:lpwui} satisfy the boundary growth condition with arbitrarily small $B_i>0$, we can apply Theorems \ref{thm:rqmcge} and \ref{thm:mcerrbound} to obtain the convergence rate of the mean generalization error for different sampling methods.
	
	\begin{thm}\label{thm:PDEerror}
		Under Framework \ref{frame:lpwui}. Suppose that the drift function $ \mu$ and the diffusion function $\sigma$ in \eqref{eq:target} satisfy Assumption \ref{assum:driftDiffusion}, then
		\begin{enumerate}
			\item[(i)] Suppose that $(\bm u_{i})_{i=1}^{\infty}$ is the scrambled digital sequence in base $b\geq 2$, let sample size $n=b^{m}$, $m\in \mathbb{N}$, then for arbitrarily small $\epsilon>0$ we have 
			\begin{equation}
				\mathbb{E}\left[ \mathcal{L}(f_{n,\mathcal{H}})-\mathcal{L}(f_{\mathcal{H}})\right]  = O(n^{-1+\epsilon}). \nonumber
			\end{equation}
			
			\item[(ii)] Suppose that $( \bm u_{i})_{i=1}^{\infty}$ are i.i.d samples simulated from the uniform distribution over $(0,1)^{d+D}$, then for arbitrarily small $\epsilon>0$ we have 
			\begin{equation}
				\mathbb{E}\left[ \mathcal{L}(f_{n,\mathcal{H}})-\mathcal{L}(f_{\mathcal{H}})\right]  = O(n^{-\frac{1}{2}+\epsilon}). \nonumber
			\end{equation}
		\end{enumerate}
	\end{thm}
	\begin{proof}
		This follows directly from Theorems \ref{thm:rqmcge}, \ref{thm:mcerrbound} and Lemma \ref{lem:checkBGC}. $\square$
	\end{proof}
	
	As discussed in Section \ref{sec:2.1}, the mean estimation error from the ERM principle $$\mathbb{E}\left[ \mathcal{E}(f_{n,\mathcal{H}})\right] =\mathbb{E}\left[(b-a)^{-d}\norm{f_{n,\mathcal{H}}-f^{*}}_{L^{2}([a,b]^{d})}^{2}\right] $$ can be decomposed into 
	\begin{equation}
		\underbrace{\mathbb{E}\left[\mathcal{L}(f_{n,\mathcal{H}})-\mathcal{L}(f_{\mathcal{H}})\right] } _{\text{mean generalization error}}+\underbrace{(b-a)^{-d}\norm{f_{\mathcal{H}}-f^{*}}_{L^{2}([a,b]^{d})}^{2}}_{\text{approximation error}}. \nonumber
	\end{equation}
	The approximation error is independent of the sample $(\bm u_{i})_{i=1}^{n}$, hence Theorem \ref{thm:PDEerror} shows that we can achieve asymptotically smaller mean estimation error as sample size $n\to \infty$ if we simulate scrambled digital nets instead of i.i.d. uniform points. 
	\section{Numerical Experiments}\label{sec:numeriExp}
	In this section, we conduct numerical experiments to show the potential advantages obtained by using scrambled digital nets instead of i.i.d. uniform random numbers in the deep learning algorithm for solving linear Kolmogorov PDEs. Due to the low computational cost of generating uniform samples on unit cube, we follow \cite{Berner:2020b,Richter2022} to simulate new uniform samples for each batch during training. To address the influence of the sampling methods, we ensure that other settings of the algorithms are same. The activation function of artificial neural network is chosen to be the swish function \eqref{eq:swish}, and we initialize the artificial neural network by means of the Xavier initialization \cite{glorot2010understanding}. Moreover we add batch normalization layers \cite{ioffe2015batch} to enhance model robustness. For training the networks, we follow the training settings from \cite{Berner:2020b}. More specifically, we use the Adam optimizer \cite{kingma:2017adam} with default weight decay $0.01$ and piece-wise constant learning rate. The detailed settings of the deep learn algorithm are given in Table~\ref{tab1}. 
	
	To compare the performance of the deep learning algorithms based on different sampling methods, we use the relative $L^{2}$ error defined by
	$$\sqrt{\frac{\int_{[a,b]^{d}}\left(\Phi_{\theta}(x)-u^{*}(T,x)\right)^{2}\mathrm{d}x}{\int_{[a,b]^{d}}\left(  u^{*}(T,x)\right)  ^{2}\mathrm{d}x}},$$
	where $\Phi_{\theta}$ is the output of the neural network after training and $u^{*}(T,\cdot)$ is the exact solution. Since the relative $L^{2}$ error cannot be computed explicitly, we approximate it via MC methods that is
	\begin{equation} \label{eq:L2}
		\sqrt{\frac{m^{-1}\sum_{i=1}^{m}\left(\Phi_{\theta}(x_{i})-u^{*}(T,\bm x_{i})\right)^{2}}
			{m^{-1}\sum_{i=1}^{m}\left( u^{*}(T,\bm x_{i})\right)^{2}}}, 
	\end{equation}
	where the sample size $m\in \mathbb{N}$ and $(\bm x_{i})_{i=1}^{m}$ are i.i.d. samples with uniform distribution on $[a,b]^{d}$.

		Deep learning algorithms typically fine-tune the parameters to obtain good results. To fairly compare two different sampling methods, we use the same hyperparameter setting listed in Table~\ref{tab2}. For each batchsize, we select the parameter that yield the smallest relative $L^{2}$ error on the validation data. We implement the experiment in PyTorch, for more details, please refer to  \href{https://github.com/JcXiao1/LinearKolmogorovPDE_RQMC}{https://github.com/JcXiao1/LinearKolmogorovPDE\_RQMC} for more details.
	
	\begin{table}[tbhp]
		\caption{Training settings}	\label{tab1}
		\begin{center}
			\begin{tabular}{lll}
				\hline
				& Heat equation & Black-Scholes model   \\
				\hline
				\textbf{Input} & &    \\
				dimension & dim = 5, 50 & dim = 5, 50 \\
				$[a,b]$ & $[0,1]$ & $[4.5,5.5]$ \\
				T & 1 & 1 \\
				\hline
				\textbf{Network} & &    \\
				width & 4$\times$dim & 4$\times$dim\\
				depth & 6 & 6\\
				activation & swish & swish  \\
				norm layer & batch norm. & batch norm. \\
				init & Xvaier & Xvaier \\
				\hline
				\textbf{Training} & &    \\
				batchsize & $2^{10},2^{11},\dots,2^{18}$ & $2^{10},2^{11},\dots,2^{18}$\\
				optimizer & AdamW & AdamW \\
				weight decay &0.01&0.01\\
				\hline
			\end{tabular}
		\end{center}
	\end{table}
	
	\begin{table}[tbhp]
		\caption{Training hyperparameter}	\label{tab2}
		\begin{center}
			\begin{tabular}{ll}
				\hline
				initial learning rate & $10^{-1},10^{-2},10^{-3} $\\
				lr decay ratio & $0.2, 0.5$  \\
				lr decay patience & $4000, 8000$ \\
				iterations & 60000 \\
				\hline
			\end{tabular}
		\end{center}
	\end{table}
	
	\subsection{Heat equation}\label{sec:4.1}
	We first consider the following heat equation with paraboloid initial condition
	\begin{equation}\label{eq:heateq}
		\begin{cases}
			\frac{\partial u}{\partial t}(t,x) = \Delta_{x}u(t,x), \ (t,x) \in [0,T]\times [a,b]^{d} ,	\\
			u(0,x)=\left\lVert x\right\rVert^{2}_{2}, \ x \in [a,b]^{d}.
		\end{cases}
	\end{equation}
	
	The exact solution $u^*(t,x)$ is given by
	\begin{equation}
		u^{*}(t,x)= \left\lVert x\right\rVert^{2}_{2}+2dt, \ (x,t) \in  [0,T]\times [a,b]^{d}. \nonumber
	\end{equation}
	In our experiments, we choose $[a,b]=[0,1]$ and $T=1$. To approximate the relative $L^{2}$ error, we choose $m=2^{16}$ in \eqref{eq:L2} and compute the exact solsution directly for both validation data and test data. 
	
	For different batchsizes in dimensions $5$ and $50$, Figures~\ref{fig:heat}(a) and \ref{fig:heat}(b) present the relative $L^{2}$ error on the test data, where the best model is selected based on validation data. Both figures demonstrate the superiority of the RQMC sampling method over the crude MC, and such superiority is more significant in dimension 5. In dimension 5, MC-based deep learning algorithm achieves $L^{2}$ relative error $ 5.287\times 10^{-4}$ with batchsize $2^{16}$ while RQMC-based one only requires batchsize $2^{11}$ to achieve relative error $5.082\times 10^{-4}$. In dimension 50, the advantages become slightly less impressive in the sense that MC-based deep learning algorithm requires batchsize $2^{13}$ while RQMC-based one only requires batchsize $2^{12}$ to achieve $L^{2}$ relative error close to $10^{-3}$. In general, as in other applications, the superiority of RQMC methods becomes weaker in higher dimensions .
	
	\begin{figure}[tbhp]
		\centering
		\subfloat[$d=5$]{
			\includegraphics[width=0.4\linewidth]{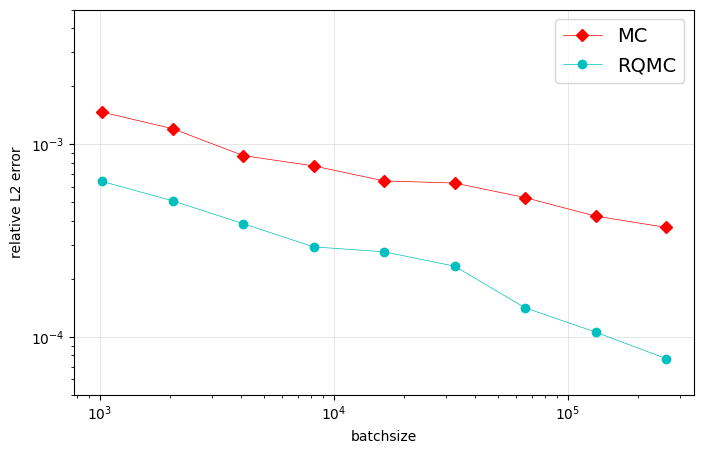}}
		\hspace{2em}
		\subfloat[$d=50$]{
			\includegraphics[width=0.4\linewidth]{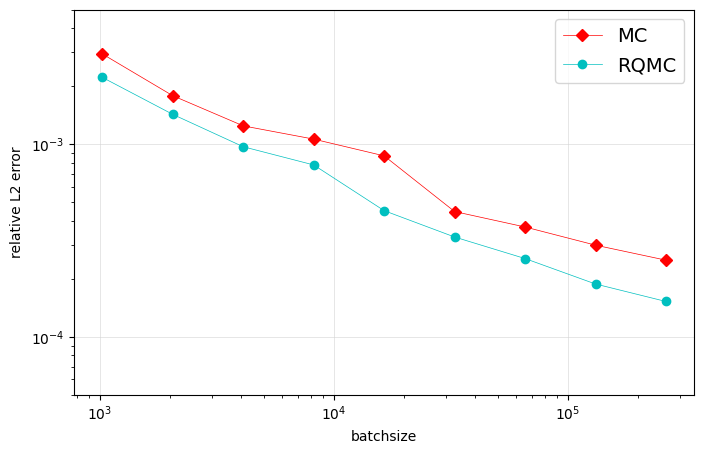}}
		\caption{Average relative $L^{2}$ error vs. batchsize  for solving heat equation in dimension 5 and 50.}
		\label{fig:heat}
	\end{figure}
	
	Figure~\ref{fig:heatTrain}(a) and \ref{fig:heatTrain}(b) present the relative $L^{2}$ error during the training process of artificial neural networks with batchsize $2^{18}$. We observe that applying RQMC sampling methods leads to a more accurate artificial neural network with smaller relative error. In $d = 5$, for the MC method, the standard deviation of the relative error over the last 100 iterations is $2.419\times 10^{-6}$, while it is only $3.210\times 10^{-7}$ for the RQMC method. In $d = 50$, the standard deviations of the relative error over last 100 iterations for RQMC and MC are $7.730\times 10^{-7}$ and $2.244\times 10^{-6}$ respectively. This suggests that the training process of the RQMC-based deep learning algorithm is more stable.
	
	\begin{figure}[tbhp]
		\centering
		\subfloat[$d=5$]{
			\includegraphics[width=0.4\linewidth]{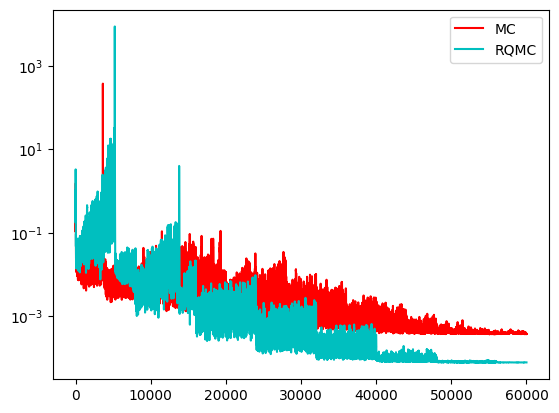}
		}
		\hspace{2em}
		\subfloat[$d=50$]{
			\includegraphics[width=0.4\linewidth]{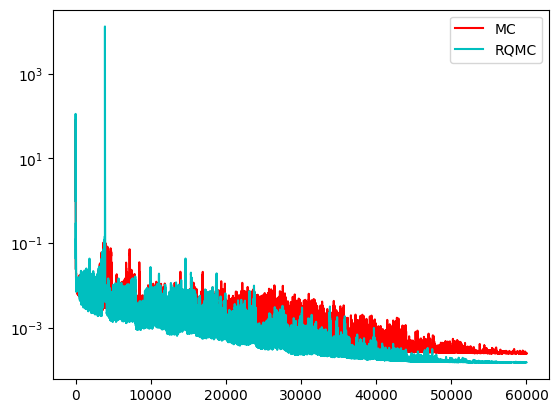}
		}
		\caption{The relative $L^{2}$ error of the training process for solving heat equation in dimension 5 and 50, the batchsize is chosen to be $2^{18}$.}
		\label{fig:heatTrain}
	\end{figure}
	
	In dimension 50, we consider the projection of the exact solution and the trained neural networks on $(0,1)^{2}$ with $x_{i} = 1/2,i=3,\dots,50$. Comparing Figures~\ref{fig:heatSolution}(a) and \ref{fig:heatSolution}(b), we find that the RQMC-based deep learning algorithm with batchsize $2^{10}$ achieves lower approximation error across the whole region than that of MC-based one. From Figures~\ref{fig:heatSolution}(b) and \ref{fig:heatSolution}(d), the RQMC-based deep learning algorithm does indeed numerically solve the heat equation \eqref{eq:heateq} with high precision on the projection space. 
	
	\begin{figure}[t]
		\centering
		\subfloat[The point-wise absolute error for MC-based deep learning with batchsize $2^{18}$.]{\label{im:heatMC}
			\includegraphics[width=0.4\linewidth]{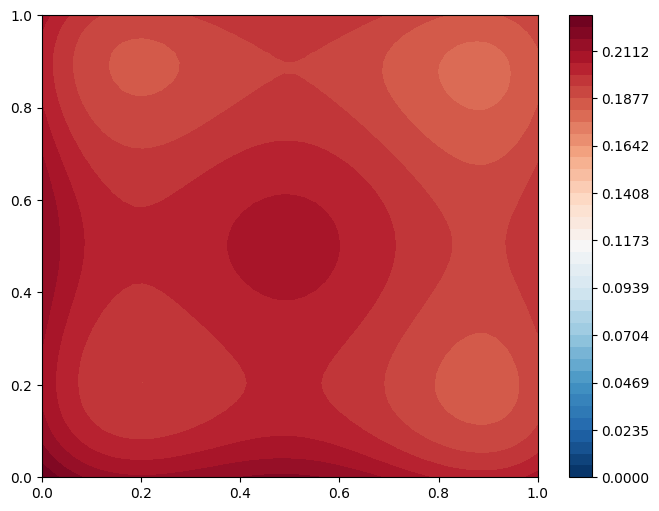}
		}
		\hspace{2em}
		\subfloat[The point-wise absolute error for RQMC-based deep learning with batchsize $2^{18}$.]{\label{im:heatQMC}
			\includegraphics[width=0.4\linewidth]{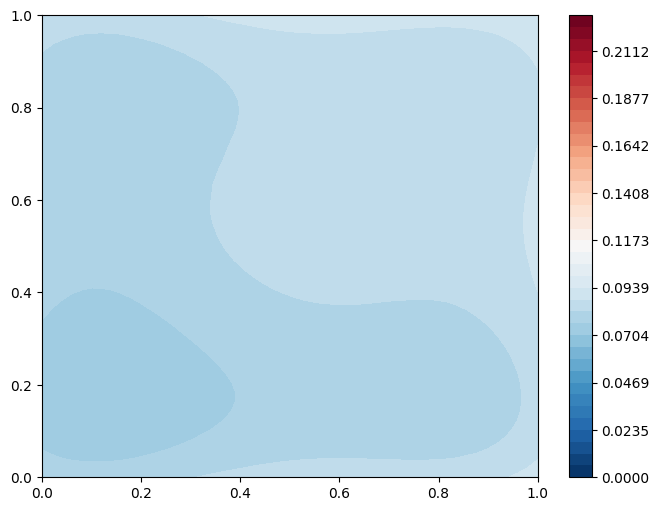}
		}
		
		\subfloat[The exact solution.]{
			\includegraphics[width=0.4\linewidth]{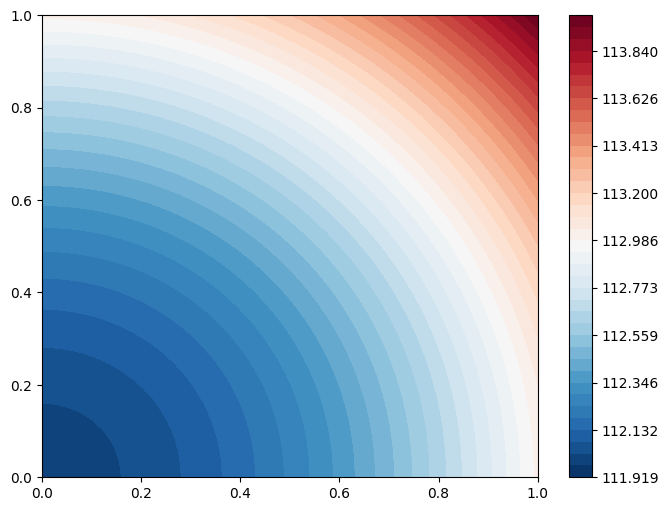}
		}
		\hspace{2em}
		\subfloat[The output function of RQMC-based deep learning with batchsize $2^{18}$.]{
			\includegraphics[width=0.4\linewidth]{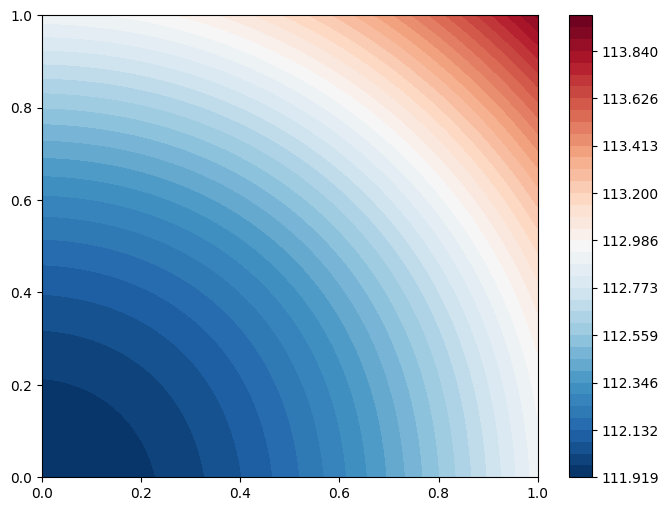}
		}
		\caption{The projection on $(0,1)^{2}$ with $x_{i} = 1/2,i=3,\dots,50$. The exact solution and the absolute error are calculated point-wisely  on a $50\times 50$ uniform grid.}
		\label{fig:heatSolution}
	\end{figure}

	\subsection{Black-Scholes model}\label{sec:4.2}
	Next, we consider the Black-Scholes PDE with correlated noise (see \cite[Section 3.4]{Beck:2018} and \cite[Section 3.2]{Richter2022}) defined by
	\begin{equation} \label{eq:BS}
		\frac{\partial u}{\partial t}(t,x) =\frac{1}{2} \sum\limits_{i = 1}^{d} \sum\limits_{j = 1}^{d} x_{i}x_{j}\left(\sigma \sigma^{T}\right)_{i,j} \frac{\partial^{2} u}{\partial x_{i}\partial x_{j}}(t,x) + \mu\sum\limits_{i = 1}^{d} x_i\frac{\partial u}{\partial x_{i}}(t,x), 
	\end{equation}
	where $\sigma \in \mathbb{R}^{d\times d}, \mu \in \mathbb{R}$. For different choice of the initial functions $\varphi(\cdot)$, Black-Scholes PDEs can model the pricing problem for different types of European options. Let $T,K \in (0,\infty)$ and the initial condition 
	\begin{equation}\label{eq:BSinit}
		\varphi(x)= \exp\left(-\mu T\right)
		\max \left\{ 0,K-\mathop{\min}\limits_{1\le i \le d}x_{i} \right\}, \ x \in [a,b]^{d}.
	\end{equation}
	Then the solution $u^{*}(t,x)$ solves the problem of pricing a rainbow European put option. By Feynman-Kac formula, we know that for every $(t,x)\in [0,T]\times [a,b]^{d}$ it holds that
	\begin{eqnarray}\label{eq:exactsolution}
		u^{*}(t,x)=\mathbb{E}\left[\varphi\left(S_{t,1}^{x},S_{t,2}^{x},\dots,S_{t,d}^{x}\right)\right] ,
	\end{eqnarray}
	where $$S_{t,i}^{x}= x_{i} \exp\left( \left( \mu-\frac{1}{2}\norm{\sigma_i}_{2}^{2}\right) t+\seq{\sigma_i,B_{t}}_{\mathbb{R}^{d}}\right), \ \ 1\leq i \leq d,$$
	$\sigma_i$ is the $i$-th row of $\sigma$ and $\set{B_{t}\mid 0\leq t\leq T}$ is a standard $d$-dimensional Brownian motion.

	In our experiments we choose the same setting in \cite[section 3.2]{Richter2022}, let $T=1$, $[a,b]=[4.5,5.5]$, $\mu=-0.05$, $K=5.5$. Let 
	$$\sigma =\mathrm{diag}(\beta_{1},\beta_{2},\dots,\beta_{d})C $$
	where $\beta_{i} = 0.1+i/(2d)$ and $C$ is the lower triangular matrix arsing from the Cholesky decomposition $CC^{*}=Q$ with $Q_{i,i}=1$ and $Q_{i,j}=0.5$, $i \neq j \in \set{1,2,\dots,d}$. To approximate the relative $L^{2}$ error, we choose $m=2^{16}$, for every sample $\bm x_{i}$, we approximate the exact solution via MC methods with sample size $M = 2^{24}$ for both validation data and test data. .

	\begin{figure}[tbhp]
		\centering
		\subfloat[$d=5$]{
			\includegraphics[width=0.4\linewidth]{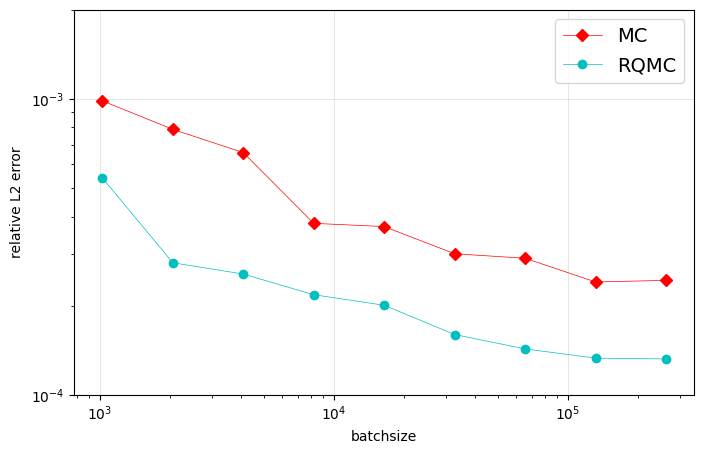}
		}
		\hspace{2em}
		\subfloat[$d=50$]{
			\includegraphics[width=0.4\linewidth]{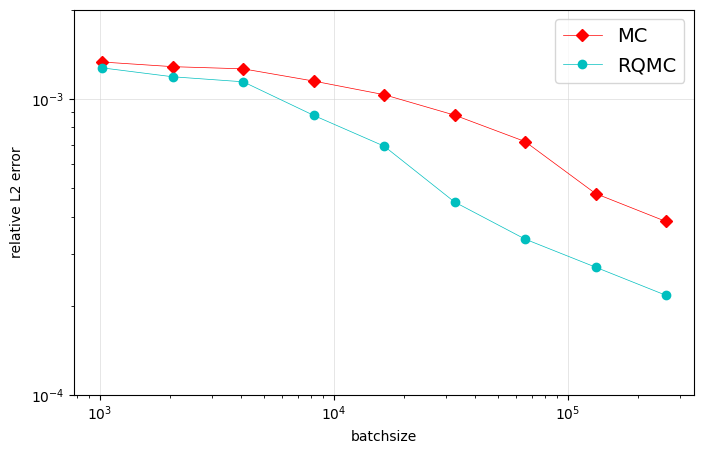}
		}
		\caption{Average relative $L^{2}$ error vs. batchsize  for solving Black-Scholes PDE in dimension 5 and 50.}
		\label{fig:BS}
	\end{figure}
	
	Same as the heat equation, Figure~\ref{fig:BS}(a) and \ref{fig:BS}(b) present the relative $L^{2}$ error on the test data in $d = 5,50$, where the best model is selected based on validation data. Both figures demonstrate the superiority of the RQMC sampling method over the ordinary MC. In dimension 5, MC-based deep learning algorithm requires batchsize $2^{16}$ to achieve $L^{2}$ relative error $2.892\times 10^{-4}$  while RQMC-based one only requires batchsize $2^{12}$ to achieve $L^{2}$ relative error $2.557\times 10^{-4}$. In dimension 50, the advantages become slightly less impressive for small batchsizes, and the MC-based deep learning algorithm requires batchsize $2^{15}$ while RQMC-based one only requires batchsize $2^{13}$ to achieve $L^{2}$ relative error less than $10^{-3}$. Figure~\ref{fig:BSTrain}(a) and \ref{fig:BSTrain}(b) present the relative $L^{2}$ error during a training process of artificial neural networks with batchsize $2^{18}$. From figures, we find that applying RQMC sampling method leads to a more accurate and stable networks.
	
	\begin{figure}[tbhp]
		\centering
		\subfloat[$d=5$]{
			\includegraphics[width=0.4\linewidth]{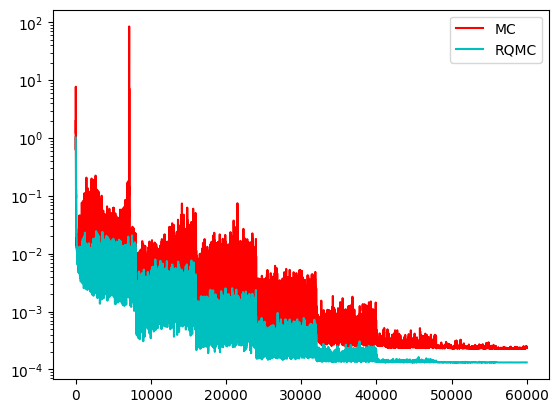}}
		\hspace{2em}
		\subfloat[$d=50$]{
			\includegraphics[width=0.4\linewidth]{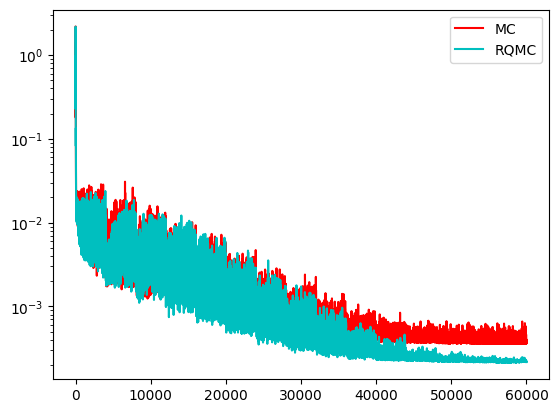}
		}
		\caption{The relative $L^{2}$ error of the training process for solving Black-Scholes PDE in dimension 5 and 50, the batchsize is chosen to be $2^{18}$.}
		\label{fig:BSTrain}
	\end{figure}
	
	In dimension 5, we consider the projection of the exact solution and the trained neural networks on $(4.5,5.5)^{2}$ with $x_{i} = 5,i=3,4,5$. Comparing Figures~\ref{fig:BSSolution}(a) and \ref{fig:BSSolution}(b), we find that the RQMC-based deep learning algorithm with batchsize $2^{15}$ achieves achieves lower approximation error across the whole region than that of MC-based one. From Figures~\ref{fig:BSSolution}(b) and \ref{fig:BSSolution}(d), the RQMC-based deep learning algorithm with batchsize $2^{18}$ is highly effective in solving the Black-Scholes PDE \eqref{eq:BS} with initial condition \eqref{eq:BSinit} on the projection space.

	\begin{figure}[tbhp]
		\centering
		\subfloat[The point-wise absolute error for MC-based deep learning with batchsize $2^{18}$.]{
			\includegraphics[width=0.4\linewidth]{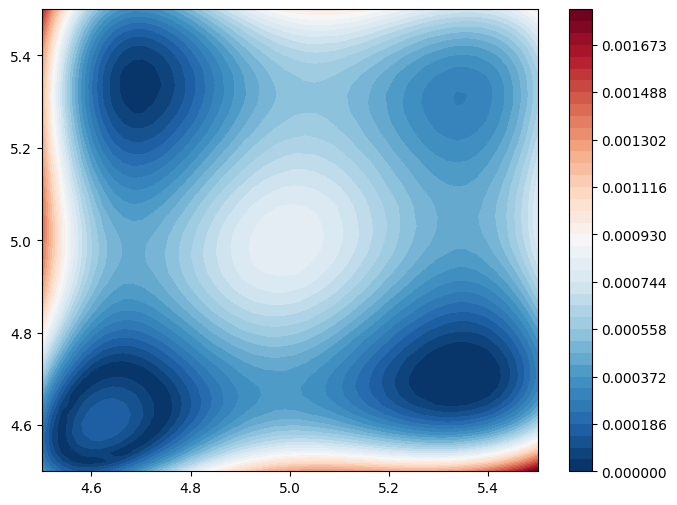}
		}
		\hspace{2em}
		\subfloat[The point-wise absolute error for RQMC-based deep learning with batchsize $2^{18}$.]{
			\includegraphics[width=0.4\linewidth]{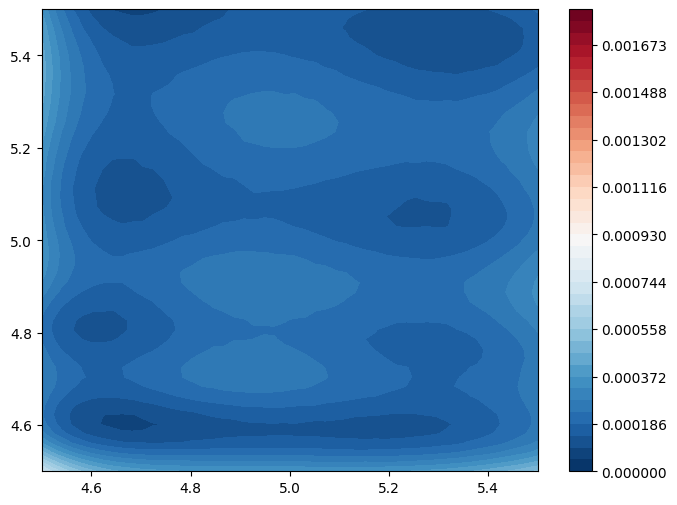}
		}
		
		\subfloat[The exact solution]{
			\includegraphics[width=0.4\linewidth]{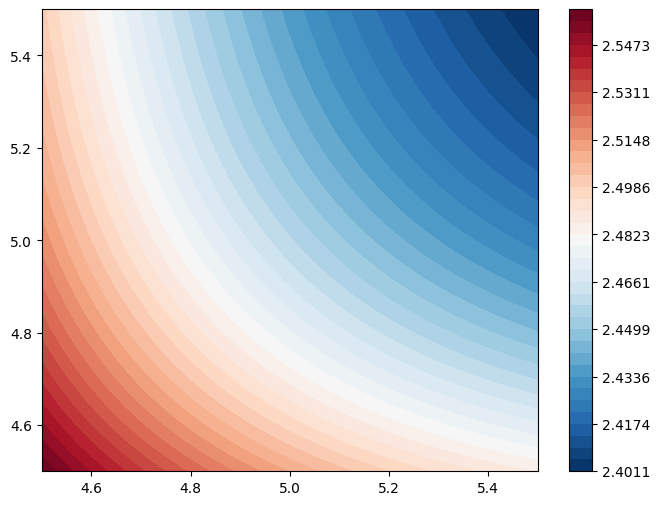}
		}
		\hspace{2em}
		\subfloat[The output function of RQMC-based deep learning with batchsize $2^{18}$.]{
			\includegraphics[width=0.4\linewidth]{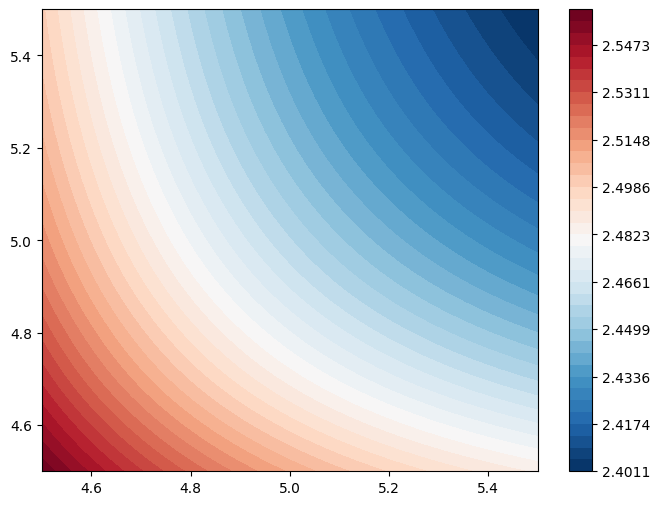}
		}
		\caption{The projection on $(4.5,5.5)^{2}$ with $x_{i} = 5,i=3,4,5$. The exact solution and absolute error are calculated point-wisely on a $50\times 50$ uniform grid. We use $2^{23}$ $RQMC$ samples to approximate the exact solution at each grid point.}
		\label{fig:BSSolution}
	\end{figure}
	\section{Conclusion}\label{sec:conclu}
	The numerical approximation of solutions to linear Kolmogorov PDEs can be reformulated as a deep learning problem under Framework \ref{frame:lpwui}. In the general framework,  the empirical loss function completely depends on the uniform data. Typically, the data are supposed i.i.d.. In this paper, we suggest that the scrambled digital sequences may be a better choice. 
	
	We decompose the error resulted from the ERM into the the generalization error and the approximation error. Since the approximation error is independent of the training data, we focus on the generalization error for different sampling strategies. For MC methods that use i.i.d. samples, we prove that the convergence rate of the mean generalization error is $O(n^{-1/2+\epsilon})$ for arbitrarily small $\epsilon>0$ if function $F$ satisfies the boundary growth condition \eqref{eq:gc} for arbitrarily small constants. For RQMC methods that use the scrambled digital sequence as the uniform data, the convergence rate of the mean generalization error becomes $O(n^{-1+\epsilon})$ for arbitrarily small $\epsilon>0$ which is asymptotically better than that of MC methods. We conduct numerical experiments to show the potential advantages obtained by replacing i.i.d. uniform data with scrambled digital sequences in the deep learning algorithm for solving linear Kolmogorov PDEs. Numerical results show that the RQMC-based deep learning algorithm outperforms the MC-based one in terms of the accuracy and robustness. 
	
	The numerical results demonstrate that we need larger batchsize to show the advantages of RQMC-based deep learning algorithm as the dimension $d$ becomes larger. One possible way to explain this phenomenon is to study on how the implied constant in the rate $O(n^{-1+\epsilon})$ and $O(n^{-1/2+\epsilon})$ depends on the dimension $d$. There is also an apparent gap between the error from the ERM and the error from stochastic gradient descent~(SDG) type algorithms such as Adam optimizer. However, the error analysis for SGD type algorithms usually relies on the convexity of the loss function, which does not actually hold for the loss function defined in Framework \ref{frame:lpwui}. For other deep learning algorithms for solving PDEs such as  Physics-informed Neural Networks and Deep Galerkin Method, If the training data can be generated with the uniform samples on unit cube, we guess that replacing the i.i.d. uniform sample with the scrambled digital net may still improve the accuracy and efficiency of the deep learning algorithms. The theoretical error analysis and relevant numerical experiments are left as future research.
	\section*{Acknowledgments}
	This work is supported by the National Natural Science Foundation of China through grant 72071119.
	
	\appendix
	\renewcommand{\thelem}{\Alph{section}\arabic{lem}}
	\renewcommand{\thedefn}{\Alph{section}\arabic{defn}}
	\section{Low variation extension}\label{appendix: lve}
	
	To prove Theorem \ref{thm:errqmc}, we need to introduce concept of the low variation extension \cite{Owen:2005,Owen:2006}. For every $\eta \in (0,1/2)$, let
	\begin{equation}\label{eq:Keta}
		K(\eta)=\set{u\in (0,1)^d\mid \underset{1\leq j\leq d}{\mathrm{\min}}\min(u_{j},1-u_{j}) \geq \eta }
	\end{equation}
	be the set avoiding the boundary of $[0,1]^d$ with distance $\eta>0$. Let $c = (1/2,\cdots,1/2)\in\mathbb{R}^d$, according to \cite[Proposition 25]{Owen:2005}, the ANOVA type decomposition of a function $g$ is given by
	\begin{equation}
		\label{eq:anova}
		g(u) = g(c) + \sum_{v\neq \emptyset}\int_{[c^v,u^v]}D_vg(z^v:c^{-v})dz^v,
	\end{equation} 
	where $z^v:c^{-v}$ denotes the point $y$ with $y_j=z_j$ for $j\in v$ and $y_j=c_j$ for $j\notin v$, $dz^v$ denotes $\prod_{i\in v}dz_i$. Based on \eqref{eq:anova}, the low variation extension $g_{\eta}(u)$ of $g(u)$ from $K(\eta)$ to $(0,1)^d$ is defined by 
	\begin{equation}
		\label{eq:lve}
		g_{\eta}(u) = g(c) + \sum_{v\neq \emptyset}\int_{[c^v,u^v]}\bm 1\{z^v:c^{-v}\in K(\eta)\}D_v g(z^v:c^{-v})dz^v.
	\end{equation}
	
	For the low variation extension \eqref{eq:lve}, Owen \cite{Owen:2006} proves some useful properties which are stated in the next lemma.
	\begin{lem}\label{lem:lveBasicproperty}
		Suppose that  $\mathcal{G}$ is a class of real-valued functions defined on $(0,1)^d$ which satisfies the boundary growth condition \eqref{eq:gc} with constants $(B_i)_{i=1}^{d}$. Let $\eta\in (0,1/2)$ and $K(\eta)$ be the set avoiding boundary defined by \eqref{eq:Keta}. For function $g\in \mathcal{G}$, let $g_{\eta}(u)$ be the low variation extension of $g(u)$ from $K(\eta)$ to $(0,1)^d$ defined by \eqref{eq:lve}. Then
		\begin{enumerate}
			\item[(i)] for every $\eta\in(0,1/2)$, $u \in K(\eta)$ and $g\in \mathcal{G}$, we have
			$$g_{\eta}(u)=g(u).$$
			\item[(ii)] there exists a constant $C_{1}>0$ such that for every $u \in (0,1)^d - K(\eta)$, $\eta\in(0,1/2)$ and $g\in \mathcal{G}$,
			\begin{equation}
				\abs{g_{\eta}(u)-g(u)}\leq C_1\prod_{i=1}^d[\min(u_i,1-u_i)]^{-B_i}. \nonumber
			\end{equation}
			\item[(iii)] there exists a constant $C_{2}>0$ such that for every $g\in \mathcal{G}$ and $\eta\in(0,1/2)$, 
			\begin{equation}
				V_{\mathrm{HK}}\left( g_{\eta}\right) \leq  C_{2}\eta^{-\max_{i}B_i}. \nonumber
			\end{equation}
		\end{enumerate}
	\end{lem}
	
	The next lemma states other properties of low variation extensions that are necessary to prove Theorems \ref{thm:errqmc} and \ref{thm:mcerrbound}.
	\begin{lem}\label{lem:lveproperty}
		Under the setting of Lemma \ref{lem:lveBasicproperty}, suppose that $ \bm u_1,\dots, \bm u_n$ are uniformly distributed on $(0,1)^d$. Then
		\begin{enumerate}
			\item[(i)] there exists a constant $C_{3}>0$ such that for every $\eta\in(0,1/2)$,
			\begin{eqnarray*}
				\E{\underset{g\in \mathcal{G}}{\mathrm{sup}}\left( \frac{1}{n}\sum_{i=1}^{n}\abs{g_{\eta}( \bm u_i)-g(\bm u_i)}\right)}  &\leq& C_3 \eta^{1-\max_{i}B_i},\\ 
				\underset{g\in \mathcal{G}}{\mathrm{sup}}\int_{(0,1)^d}\abs{ g_{\eta}( u)-g( u)}d  u  &\leq& C_3 \eta^{1-\max_{i}B_i}.
			\end{eqnarray*}
			\item[(ii)] there exists a constant $C_{4}>0$ such that for every $\eta\in(0,1/2)$ and $g\in \mathcal{G}$,
			\begin{equation}
				\underset{u \in (0,1)^d}{\mathrm{sup}}\abs{g_{\eta}(u)} \leq C_4 \eta^{-\sum_{i=1}^{d}B_i}. \nonumber
			\end{equation}
		\end{enumerate}
	\end{lem}
	\begin{proof}
		Define $m(x)=\min(x,1-x)$ for $x\in \mathbb{R}$. From $(ii)$ of Lemma \ref{lem:lveBasicproperty}, we have
		\begin{eqnarray*}
			\underset{g\in \mathcal{G}}{\mathrm{sup}}\left( \frac{1}{n}\sum_{i=1}^{n}\abs{g_{\eta}(\bm u_i)-g( \bm u_i)}\right)  &\leq& \frac{1}{n}\sum_{i=1}^{n}\underset{g\in \mathcal{G}}{\mathrm{sup}}\abs{g_{\eta}(\bm u_i)-g( \bm u_i)}  \\
			&\leq&\frac{C_1}{n}\sum_{i=1}^{n}\prod_{j=1}^d[m( u_{i,j})]^{-B_j}1\{\bm u_i \notin K(\eta) \}.
		\end{eqnarray*}
		Taking expectations on both sides of the above inequality, then it suffices to show that for every $1\leq i \leq n$,
		\begin{eqnarray*}
			&&C_1\E{\prod_{j=1}^d[m( u_{i,j})]^{-B_j}1\{\bm  u_{i} \notin K(\eta) \}} \\
			&\leq& C_1 \sum_{j=1}^{d} \int_{[0,\eta]\cup[1-\eta,1]}[m( u_{i,j})]^{-B_j}d  u_{i,j}\int_{(0,1)^{\set{-j}}}\prod_{k \in \set{-j}}[m( u_{i,j})]^{-B_k}d  u_{i,k} \\
			&=&C_1\sum_{j=1}^{d}\frac{2\eta^{1-B_{j}}}{1-B_{j}}\prod_{k \in \set{-j}}\frac{2^{B_{k}}}{1-B_{k}} \leq 2C_1\left( \prod_{k=1}^d\frac{2^{B_{k}}}{1-B_{k}}\right) \sum_{j=1}^{d}\eta^{1-B_{j}}\\
			&\leq& 2dC_1\left( \prod_{k=1}^d\frac{2^{B_{k}}}{1-B_{k}}\right) \eta^{1-\max_{i}B_{i}} = C_2 \eta^{1-\max_{i}B_{i}},
		\end{eqnarray*}
		where the first inequality follows from extending the integration region to $(0,1)^{\set{-j}}$ for the rest arguments when $j$-th argument is fixed. It is easy to use the same techniques to prove that
		\begin{eqnarray*}
			\underset{g\in \mathcal{G}}{\mathrm{sup}}\int_{(0,1)^d}\abs{ g_{\eta}( u)-g( u)}d  u &\leq& 	\int_{(0,1)^d}C_{1}\prod_{j=1}^d[m( u_{j})]^{-B_j}1\{ u \notin K(\eta) \}d u \\
			&\leq& C_2 \eta^{1-\max_{i}B_{i}}.
		\end{eqnarray*}
		
		For $(ii)$, from the expression \eqref{eq:lve} and the growth condition \eqref{eq:gc}, we have
		\begin{eqnarray*}
			\abs{g_{\eta}(u)} &\leq& \abs{g(c)}+\sum_{v\neq \emptyset}\int_{K(\eta)}\abs{D_vg(z^v:c^{-v})}dz^v \\
			&\leq&\abs{g(c)}+B\sum_{v\neq \emptyset}2^{\sum_{j\notin v}B_j}\prod_{j\in v}\int_{\eta}^{1-\eta}[m(z_{j})]^{-B_j-1}dz_j\\
			&\leq&B2^{\sum_{j=1}^dB_j}+B2^{\sum_{j=1}^dB_j}\sum_{v\neq \emptyset}\prod_{j\in v}\frac{\eta^{-B_j}}{B_j} \\
			&=&B2^{\sum_{j=1}^dB_j}\prod_{j=1}^d\left(  \frac{\eta^{-B_j}}{B_j}+1\right) \\
			&\leq& B\left[ \prod_{i=1}^{d}2^{B_j}\left( 1+\frac{1}{B_{j}}\right) \right] \eta^{-\sum_{j=1}^{d}B_j}= C_3 \eta^{-\sum_{j=1}^{d}B_j}.
		\end{eqnarray*}
		
		The proof is completed. $\square$
	\end{proof}
	
	\section{Rademacher complexity}\label{appendix:rc}
	
	\begin{defn}[Rademacher complexity]\label{defn:rc}
		Suppose $\set{\sigma_i}_{i=1}^{n}$ are i.i.d Rademacher variables defined on the probability space $(\Omega,\mathcal{F},\mathbb{P})$ with discrete distribution $\mathbb{P}\{\sigma_i=1\}=\mathbb{P}\{\sigma_i=-1\}=1/2$. Suppose that $( \bm x_{i})_{i=1}^{n}$ are i.i.d random samples on $(\Omega,\mathcal{F},\mathbb{P})$ and take values in a metric space $S$. Suppose $\mathcal{G}$ is a class of real-valued borel-measurable function over $(S,\mathcal{B}(S))$ ($\mathcal{B}(S)$ is the Borel $\sigma$-field). The Rademacher complexity of a function class $\mathcal{G}$ with respect to the random sample $(\bm x_{i})_{i=1}^{n}$ is defined by
		\begin{equation}
			R_{n}(\mathcal{\mathcal{G}};( \bm x_{i})_{i=1}^{n}) = \E{\underset{g \in \mathcal{G}}{\mathrm{sup}}\frac{1}{n}\sum_{i=1}^n\sigma_i g(\bm  x_{i})}. \nonumber	
		\end{equation}
	\end{defn}
	
	The following lemmas state some useful properties of the Rademacher complexity which are applied to bound the mean generalization error.
	\begin{lem}\label{lem:rcbasic}
		Under the setting of Definition \ref{defn:rc}, we have
		\begin{equation}
			\E{\underset{g \in \mathcal{G}}{\mathrm{sup}}\abs{\frac{1}{n}\sum_{i=1}^n g( \bm x_{i})-\E{g( x)}}} \leq 2R_{n}(\mathcal{\mathcal{G}};( \bm x_{i})_{i=1}^{n}) .\nonumber
		\end{equation}
	\end{lem}

	\begin{lem}\label{lem:rcbound}
		Under the setting of Definition \ref{defn:rc}. Suppose $w$ is a bounded measurable real-valued function over $(S,\mathcal{B}(S))$, then we have 
		\begin{equation}
			R_{n}(w\cdot\mathcal{G};(\bm x_{i})_{i=1}^{n})\leq \underset{x \in S}{\mathrm{sup}}\abs{w(x)}R_{n}(\mathcal{G};( \bm x_{i})_{i=1}^{n}),\nonumber		\end{equation}
		where $w\cdot\mathcal{G}$ denotes the function class $\set{w\cdot g\mid g\in \mathcal{G}}$.
	\end{lem}   
	
	For the rigorous proofs of Lemmas \ref{lem:rcbasic} and \ref{lem:rcbound}, we refer the readers to \cite[Lemma~26.2]{Shalev:2014} and \cite[Lemma5.2]{Jiao:2021}, respectively. Under Framework \ref{frame:lpwui}, the function class $\mathcal{G}$ is the restriction of an artificial neural networks over $S = [a,b]^{d}$. With reference to \cite[Theorem 5.13]{Jiao:2021}, the next two lemmas provide the upper bound for the specific Rademacher complexity.
	
	\begin{lem}\label{lem:RandANN}
		Under the setting of Definition \ref{defn:rc}. Let $p \in \mathbb{N}$, $B>0$, let $\Theta \subseteq \mathbb{R}^{p}$ and $\norm{\theta}_{2}\leq B$ for $\theta\in \Theta$. Suppose that there exists a surjective operator $\mathcal{R}:(\Theta,\norm{\cdot}_{2}) \to \mathcal{G}$ with image $\mathcal{G}$. Suppose that there exist constants $B_{1}$ and $L_{1}$ such that for all $\theta_{1},\theta_{2}\in \Theta$, $g\in \mathcal{G}$ it holds that
		\begin{eqnarray*}
			\underset{x \in S}{\mathrm{sup}}\abs{\mathcal{R}(\theta_{1})(x)-\mathcal{R}(\theta_{2})(x)} 
			&\leq& L_{1}\norm{\theta_{1}-\theta_{2}}_{2}, \\
			\underset{x \in S}{\mathrm{sup}}\abs{g(x)}
			&\leq& B_{1}.
		\end{eqnarray*}
		Then we have
		\begin{eqnarray*}
			R_{n}(\mathcal{G};( x_{i})_{i=1}^{n})&\leq& \frac{4}{\sqrt{n}}+\frac{6\sqrt{p}B_{1}}{\sqrt{n}}\sqrt{\log\left(2L_{1}B\sqrt{n} \right) }\\
			&=&O(n^{-\frac{1}{2}}\left( \log n\right) ^{\frac{1}{2}}).
		\end{eqnarray*}
	\end{lem}
	
	\begin{lem}\label{lem:rc_neural}
		Suppose that $\mathcal{H}=\mathcal{N}^{a,b}_{\rho,R,S}$ and $x_i\in[a,b]^d$ for every $i \in \{1,2,\cdots,n\}$, then we have
		\begin{equation}\label{eq:rcH}
			R_{n}(\mathcal{H};( x_{i})_{i=1}^{n})=O(n^{-\frac{1}{2}}\left( \log n\right) ^{\frac{1}{2}}), \nonumber
		\end{equation}
		similarly, we can prove that 
		\begin{equation}\label{eq:rcH2}
			R_{n}(\mathcal{H}^{2};( x_{i})_{i=1}^{n})=O(n^{-\frac{1}{2}}\left( \log n\right) ^{\frac{1}{2}}), \nonumber
		\end{equation}
		where $\mathcal{H}^2$ denotes the function class $\set{h^2\mid h\in \mathcal{H}}$.
	\end{lem}
	\begin{proof}
		For $\mathcal{H}=\mathcal{N}^{a,b}_{\rho,R,S}$, the surjective operator $\mathcal{R}$ in Lemma \ref{lem:RandANN} is defined by
		$$\mathcal{R}:\theta \in \Theta_{R} \to	W_{L}\circ\rho \circ W_{L-1}\circ\rho\cdots\circ\rho\circ W_{1}(x)|_{x\in [a,b]^{d}},$$
		see details in Definition \ref{dfn:ann}. From the smoothness of swish function $\rho$, it is obvious that $\mathcal{R}(\theta)(x)\in C^{\infty}(\Theta_{R}\times[a,b]^{d},\mathbb{R})$, thus there exist constant $B$ and $L$ such that for any $\theta_{1},\theta_{2}\in \Theta_{R}$
		\begin{equation}\label{eq:neural_property}
			\underset{x \in [a,b]^{d}}{\mathrm{sup}}\abs{\mathcal{R}(\theta_{1})(x)-\mathcal{R}(\theta_{2})(x)} \leq L\norm{\theta_{1}-\theta_{2}}_{2}, \quad \underset{x \in [a,b]^{d}}{\mathrm{sup}}\abs{\mathcal{R}(\theta_{1})(x)} \leq B.
		\end{equation}
		
		The proof is completed by applying Lemma \ref{lem:RandANN}.
	\end{proof}

\end{document}